\documentclass[11pt]{article}
\usepackage[a4paper,margin=30mm]{geometry}
\usepackage{amsmath,amssymb,amsthm,mathtools}
\usepackage{microtype}
\usepackage[numbers,sort&compress]{natbib}
\usepackage[hidelinks]{hyperref}
\hypersetup{
  pdftitle={Sharp Typical Distance and Exponential Small-Ball Bounds in One-Step-Cliff Nielsen Geometry},
  pdfauthor={Honghuai Fang},
  pdfsubject={Sharp Haar-typical distance in Nielsen quantum complexity geometry},
  pdfkeywords={quantum complexity geometry, right-invariant metrics, Haar-typical distance, Jacobi fields, small-ball probabilities, Weyl integration}
}
\usepackage[T1]{fontenc}
\usepackage{lmodern}
\allowdisplaybreaks

\newtheorem{theorem}{Theorem}[section]
\newtheorem{lemma}[theorem]{Lemma}
\newtheorem{proposition}[theorem]{Proposition}
\newtheorem{corollary}[theorem]{Corollary}

\newcommand{\PU}{\operatorname{PU}}
\newcommand{\Vol}{\operatorname{Vol}}
\newcommand{\Tr}{\operatorname{Tr}}
\newcommand{\diam}{\operatorname{diam}}
\newcommand{\spec}{\operatorname{spec}}
\newcommand{\ad}{\operatorname{ad}}
\newcommand{\Ad}{\operatorname{Ad}}
\newcommand{\op}{\mathrm{op}}
\newcommand{\HS}{\mathrm{HS}}
\newcommand{\Sone}{S_1}
\newcommand{\sinc}{\operatorname{sinc}}
\newcommand{\E}{\mathbb E}
\newcommand{\Prob}{\mathbb P}
\newcommand{\cL}{\mathcal L}
\newcommand{\cP}{\mathcal P}
\newcommand{\rank}{\operatorname{rank}}
\newcommand{\sym}{\operatorname{sym}}
\newcommand{\Arg}{\operatorname{Arg}}
\newcommand{\diag}{\operatorname{diag}}
\newcommand{\vol}{\operatorname{vol}}

\title{Sharp Typical Distance and Exponential Small-Ball Bounds\\
in One-Step-Cliff Nielsen Geometry}
\author{Honghuai Fang\thanks{Institute for Theoretical Sciences,
Westlake University, 600 Dunyu Road, Xihu District, Hangzhou,
Zhejiang 310030, China.
Corresponding author: \texttt{fanghonghuai@westlake.edu.cn}.}}
\date{}

\begin{document}
\maketitle

\begin{abstract}
Let $D=2^n$ and equip $\PU(D)$ with the one-step-cliff Nielsen metric,
with quadratic metric coefficients one in Pauli directions of weights one
and two and $D^2$ in all higher weights.  We prove that the distance from the identity of
a Haar-random element, normalized by $D$, converges to $\pi/\sqrt3$ in
probability and in $L^p$ for every $1\le p<\infty$.  Quantitatively, it lies within
$O(D^{-1/8}(\log D)^{1/2})$ of this limit outside a set of Haar measure at
most $\exp\{-\Omega(D^{7/4}\log D)\}$.  For each fixed
$0<x<\pi/\sqrt3$, the ball of radius $xD$ has Haar measure
$\exp\{-\Theta_x(D^2)\}$; for $x>\pi/\sqrt3$, its complement has measure at most
$e^{-c_xD^2}$ for some $c_x>0$.  As $x\uparrow\pi/\sqrt3$, the lower and upper
logarithmic rates are both asymptotic to
$(\pi^2/3-x^2)^2/(16\zeta(3))$.

The small-ball upper bound follows from a comparison of Jacobi determinants,
obtained by rescaling the linearized geodesic equations and applying Kato
transport.  Weyl integration reduces the remaining integral to an
Abel-regularized logarithmic-energy estimate on the circle.  A centered
principal logarithm and concentration of the circular unitary ensemble
eigenangle second moment give the distance upper bound.
\end{abstract}



\section{Introduction}\label{sec:intro}

Nielsen's geometric approach to quantum complexity assigns a cost to a
Hamiltonian according to the number of qubits on which it acts
\cite{NielsenQIC,NielsenScience,NielsenPRA,DowlingNielsen}.  Quantum evolutions
then become paths in a right-invariant metric on the unitary group.  Since
one- and two-qubit gates can be implemented by paths of bounded length,
metric lower bounds give circuit lower bounds.  We study the one-step-cliff
metric on the projective unitary group $\PU(D)=U(D)/U(1)$, where $D=2^n$.
The quadratic metric coefficients are one in Pauli directions of weights
one and two, and $D^2$ in all higher weights.  Our main result determines the
first-order distance from the identity of a Haar-random projective unitary.

The candidate constant is suggested by a simple path.  Choose a representative
$U\in U(D)$ with principal eigenangles $\theta_1,\ldots,\theta_D\in(-\pi,\pi]$,
and subtract their mean from its Hermitian logarithm.  The resulting
traceless matrix generates a path from $[I]$ to $[U]$.  Since no quadratic
metric coefficient exceeds $D^2$, this path has length at most
\[
 D\left(\frac1D\sum_{j=1}^D\theta_j^2\right)^{1/2}.
\]
For a Haar unitary, the empirical eigenvalue measure converges to uniform
arc length \cite{HiaiPetz}, and
\[
 \frac1{2\pi}\int_{-\pi}^{\pi}\theta^2\,d\theta=\frac{\pi^2}{3}.
\]
This gives the typical upper bound $(\pi/\sqrt3+o(1))D$.  The difficulty is
in the reverse inequality.  A minimizing path may change its Hamiltonian
throughout the evolution and use the low-weight Pauli directions to shorten
the route.  We prove that, with probability tending to one as $D\to\infty$, these
directions cannot reduce the length by a fixed positive fraction of $D$.

Let $d_A$ denote the distance and $\mu_D$ normalized Haar measure.  The
low-weight Pauli subspace has dimension
\[
 M=3n+9\binom n2=O((\log D)^2),
 \qquad
 \eta_D=\left(\frac{M}{D^{1/2}}\right)^{1/4}.
\]
Theorem~\ref{thm:sharp} states that, for suitable absolute constants $c,C>0$,
\[
 \Prob\!\left\{
 \left|\frac{d_A([I],[U])}{D}-\frac{\pi}{\sqrt3}\right|>C\eta_D
 \right\}
 \le 2e^{-cM^{1/2}D^{7/4}}.
\]
In particular, the normalized distance converges to $\pi/\sqrt3$ in
probability and in every finite $L^p$, $1\le p<\infty$.
Theorem~\ref{thm:ball-probabilities} gives the corresponding fixed-radius
estimates: for $0<x<\pi/\sqrt3$,
\[
 \mu_D\bigl(B_A([I],xD)\bigr)=e^{-\Theta_x(D^2)},
\]
whereas for $x>\pi/\sqrt3$ the complement has measure at most
$e^{-c_xD^2}$.  Thus $D^2$ is the optimal large-deviation speed below the
threshold.  The lower and upper logarithmic rates have the common
near-threshold asymptotic
\[
 \frac{(\pi^2/3-x^2)^2}{16\zeta(3)}
 \qquad (x\uparrow\pi/\sqrt3),
\]
as stated in Corollary~\ref{cor:local-rate}.

\subsection*{Relation to prior work}

Brown used Ricci curvature estimates and the Bishop--Gromov comparison theorem
to obtain exponential lower bounds for typical Nielsen complexity under broad
classes of penalty schedules \cite{BrownBG}.  With his qubit number equal to
$n$ and his cliff parameter equal to $D^2$, the staircase-enhanced bound
\cite[Eq.~(2.6)]{BrownBG}, which suppresses factors subexponential in $n$,
gives a lower bound of the form
\[
 C_{\mathrm{typical}}\ge 2^{n-o(n)}=D^{1-o(1)}.
\]
This determines the exponential growth rate in $n$, rather than a positive
limiting value of $C_{\mathrm{typical}}/D$.  We determine that value and prove
fixed-radius probability estimates on both sides of it.  The central
$U(1)$ direction has bounded diameter in Brown's normalization and does not
affect this exponential-scale comparison with $\PU(D)$.

Dowling and Nielsen derived the connection, geodesic equation, curvature,
and Jacobi equation for Nielsen metrics \cite{DowlingNielsen}.
Rios Ribeiro and Trancanelli studied multiple cost factors and the resulting
families of conjugate points \cite{RibeiroTrancanelli}.  Our lower bound uses
the full time-dependent Jacobi propagator to control the volume of a metric
ball.  The conjugate-time upper estimate of Le Brigant, Lichtenfelz, and
Preston \cite[Corollary~3.5]{LBLP} restricts the phase range along a minimizing
geodesic.  Section~\ref{sec:jacobi-dynamics} includes a proof of the required
estimate, beginning with the second variation of energy in the right-invariant
convention.

The random-matrix input is the eigenvalue law of the circular unitary ensemble
(CUE).  Hiai and Petz proved its empirical-measure large-deviation principle
at speed $D^2$ \cite{HiaiPetz}.  Diaconis and Evans proved Gaussian limits for
linear statistics satisfying a Fourier summability condition
\cite[Theorem~5.1]{DiaconisEvans}.  An Abel estimate for the circle logarithmic
kernel gives the finite-$D$ second-moment concentration used here.  The
large-deviation principle supplies the matching probability lower bounds,
and the linear-statistic theorem gives the fluctuations of the logarithmic
path associated with a Haar representative in
Appendix~\ref{app:logpath-fluctuations}.

The metric result should be distinguished from exact circuit synthesis.
Parameter counting gives an $\Omega(D^2)$ lower bound on the number of
continuous two-qubit gates needed for a Haar-generic exact synthesis;
matching-order constructions under connected qubit-connectivity constraints
are given in \cite{YuanCircuit}.  Corollary~\ref{cor:circuit} instead gives an
$\Omega(D)$ lower bound stable under a fixed projective operator-norm error,
with an explicit exceptional probability.  Brown's polynomial comparisons
between circuit models and complexity geometries address a broader range of
metrics \cite{BrownPoly}.

\subsection*{Outline of the proof}

The main task is to bound the volume swept out by minimizing geodesics.
Although the low-weight subspace has dimension only $M$, a low-rank
perturbation at each time need not produce a low-rank change in the endpoint
Jacobi map: its range can move during the evolution.  We therefore estimate
the accumulated effect of this motion rather than freeze the instantaneous
curvature spectrum.

In the frame of the conserved momentum, the inverse inertia is
\[
 G(t)=P(t)+D^{-2}P(t)^\perp,
\]
where $P(t)$ is an orthogonal projection of rank $M$.  Scaling position and
momentum variations by opposite quarter powers of $G(t)$ balances the terms
coming from $P'(t)$ and the off-diagonal commutator.  Their integrated trace
norm over a time interval of length $O(D)$ is $O(MD^{3/2})$.
Kato transport then permits comparison with a linear system having constant
coefficients.  The determinant estimate in Theorem~\ref{thm:output}
converts these estimates into a determinant bound even when the reference
Jacobi map is singular.

Next, a Gaussian density comparison transfers the average over the
high-weight Pauli subspace to the full traceless Hermitian sphere.  The
squared Vandermonde in Weyl's formula cancels the denominators of the sinc
factors of the constant-coefficient model.  What remains is a regularized
chord-length product on the
circle.  The conjugate-time estimate bounds its lifted phase range, and the
momentum normalization bounds its principal-angle second moment.  A deficit
$\Delta$ from $\pi^2/3$ then contributes a negative logarithmic-volume term
of order $\Delta^2D^2$.

The quantitative scale comes from comparing this term with the Jacobi error.
The Abel estimate uses a relative singular-value cutoff $\rho\asymp\Delta^2$,
so the two leading terms in the logarithmic bound are
\[
 -c\Delta^2D^2
 \qquad\text{and}\qquad
 C\Delta^{-2}MD^{3/2}.
\]
The negative term dominates when $\Delta$ is a sufficiently large constant
multiple of $(M/D^{1/2})^{1/4}$.  This yields the window $\eta_D$ and the
exceptional exponent $M^{1/2}D^{7/4}$.  The same circle estimate controls
both signs of the CUE second-moment deviation and, with the logarithmic
path, gives the matching distance upper bound.

Sections~\ref{sec:jacobi-dynamics}--\ref{sec:abel-weyl} establish the geometric
and integral comparisons.  Section~\ref{sec:small-ball} proves the lower-tail
estimates.  Section~\ref{sec:cue-upper} combines them with the CUE bounds and
derives bounds for the lower and upper logarithmic rates.
Section~\ref{sec:circuit-outlook} proves the
circuit consequence and discusses the remaining geometric questions.

\section{Model and main results}\label{sec:model-results}

Throughout, $D=2^n$ and $n\to\infty$, so $D$ ranges over powers of two.
Let $\PU(D)=U(D)/U(1)$, write $[U]$ for the projective class of
$U\in U(D)$, and let $[I]$ be the identity class.  Identify
$\mathfrak g=\mathfrak{pu}(D)$ with the traceless skew-Hermitian matrices and
equip it with the normalized Hilbert--Schmidt inner product
\begin{equation}\label{eq:hs}
 \langle X,Y\rangle_0=D^{-1}\Tr(X^\dagger Y),
 \qquad \|X\|_0=\langle X,X\rangle_0^{1/2}.
\end{equation}
We use the convention $\sinc(t)=\sin(t)/t$ for $t\ne0$ and $\sinc(0)=1$.

For $\boldsymbol\alpha=(\alpha_1,\ldots,\alpha_n)\in\{0,1,2,3\}^n$, let
$W_{\boldsymbol\alpha}=\sigma_{\alpha_1}\otimes\cdots\otimes
\sigma_{\alpha_n}$, where $\sigma_0=I_2$ and
$\sigma_1,\sigma_2,\sigma_3$ are the usual Pauli matrices.  Its Pauli
weight is
\[
 |\boldsymbol\alpha|=\#\{j: \alpha_j\ne0\}.
\]
The matrices $iW_{\boldsymbol\alpha}$ with $\boldsymbol\alpha\ne0$ form
an orthonormal basis of $\mathfrak g$.  Let $P$ be the orthogonal projection
onto the span of the basis elements of weights one and two.  For every
orthogonal projection $R$ below, write $R^\perp=I-R$.  Thus
$P\mathfrak g$ and $P^\perp\mathfrak g$ are the low-weight and
high-weight Pauli subspaces, respectively.  The rank of $P$ is
\begin{equation}\label{eq:M}
 M=3n+9\binom n2=O(n^2)=O((\log D)^2).
\end{equation}
Put $N=D^2-1$, set
\[
 S_0^{N-1}=\{X\in\mathfrak g: \|X\|_0=1\},
\]
and let $Q=D^2$.  Define the one-step-cliff inertia
\begin{equation}\label{eq:inertia}
 A=P+QP^\perp,\qquad G=A^{-1}=P+qP^\perp,\qquad q=Q^{-1}=D^{-2}.
\end{equation}
The corresponding inner product and norm at the identity are
$\langle X,Y\rangle_A=\langle X,AY\rangle_0$ and
$\|X\|_A=\langle X,X\rangle_A^{1/2}$, extended right-invariantly.  Geodesic speeds
are measured in the $A$-metric.  Let $d_A$ and
$\Vol_A$ denote the resulting distance and Riemannian volume.  For
$p\in\PU(D)$ and $r>0$, define the open metric ball
\[
 B_A(p,r)=\{q\in\PU(D):d_A(p,q)<r\}.
\]
For a piecewise $C^1$ path $\gamma:[a,b]\to\PU(D)$, write
$\ell_A(\gamma)=\int_a^b\|\dot\gamma(t)\|_A\,dt$.  We write
$\Vol_0$ for the Riemannian volume
associated with $\langle\cdot,\cdot\rangle_0$.  For every measurable
$E\subseteq\PU(D)$, define
\begin{equation}\label{eq:haar-definition}
 \mu_D(E):=\frac{\Vol_A(E)}{\Vol_A\PU(D)}.
\end{equation}
Because $\Vol_A$ is right invariant, $\mu_D$ is normalized Haar measure on
$\PU(D)$.  For $[U],[V]\in\PU(D)$, define the projective operator-norm
distance
\[
 \delta_{\op}([U],[V])=\inf_{\phi\in\mathbb R}
 \|U-e^{i\phi}V\|_{\op},
\]
where $\|\cdot\|_{\op}$ is the usual operator norm on $\mathbb C^D$.

\begin{theorem}[Sharp Haar-typical distance]\label{thm:sharp}
Set
\begin{equation}\label{eq:eta-sharp}
 \eta_D=\left(\frac{M}{D^{1/2}}\right)^{1/4}.
\end{equation}
There are positive absolute constants $c_{\rm sh},C_{\rm sh},D_{\rm sh}$ such
that, for every power of two $D\ge D_{\rm sh}$ and Haar-random
$[U]\in\PU(D)$,
\begin{equation}\label{eq:sharp-concentration}
 \Prob\!\left\{
 \left|\frac{d_A([I],[U])}{D}-\frac{\pi}{\sqrt3}\right|
 >C_{\rm sh}\eta_D
 \right\}
 \le2\exp\{-c_{\rm sh}M^{1/2}D^{7/4}\}.
\end{equation}
Consequently,
\begin{equation}\label{eq:sharp-convergence}
 \frac{d_A([I],[U])}{D}
 \longrightarrow\frac{\pi}{\sqrt3}
\end{equation}
in probability and in $L^p$ for every fixed $1\le p<\infty$.
\end{theorem}

\begin{theorem}[Fixed-radius ball probabilities]\label{thm:ball-probabilities}
For every fixed $0<x<\pi/\sqrt3$, there are constants
$0<c_x\le C_x<\infty$ and $D_x$ such that, for every power of two $D\ge D_x$,
\begin{equation}\label{eq:two-sided-speed}
 e^{-C_xD^2}
 \le\mu_D\!\left(B_A([I],xD)\right)
 \le e^{-c_xD^2}.
\end{equation}
For every fixed $x>\pi/\sqrt3$, there are constants $c_x^+>0$ and $D_x^+$
such that, for every power of two $D\ge D_x^+$,
\begin{equation}\label{eq:upper-threshold}
 1-\mu_D\!\left(B_A([I],xD)\right)\le e^{-c_x^+D^2}.
\end{equation}
Consequently,
\begin{equation}\label{eq:zero-one-threshold}
 \lim_{\substack{D\to\infty\\D=2^n}}
 \mu_D\!\left(B_A([I],xD)\right)
 =\begin{cases}
 0,&0<x<\pi/\sqrt3,\\
 1,&x>\pi/\sqrt3.
 \end{cases}
\end{equation}
In particular, the speed $D^2$ is optimal for every fixed subcritical radius.
\end{theorem}

\begin{corollary}
\label{cor:local-rate}
For $0<x<\pi/\sqrt3$, define the lower and upper logarithmic rates
\begin{equation}\label{eq:geometric-rate-liminf}
 \underline{I}(x)=\liminf_{\substack{D\to\infty\\D=2^n}}
 -\frac1{D^2}\log\mu_D\!\left(B_A([I],xD)\right),
 \qquad
 \overline{I}(x)=\limsup_{\substack{D\to\infty\\D=2^n}}
 -\frac1{D^2}\log\mu_D\!\left(B_A([I],xD)\right).
\end{equation}
Then
\begin{equation}\label{eq:local-rate-coefficient}
 \lim_{x\uparrow\pi/\sqrt3}
 \frac{\underline I(x)}{(\pi^2/3-x^2)^2}
 =
 \lim_{x\uparrow\pi/\sqrt3}
 \frac{\overline I(x)}{(\pi^2/3-x^2)^2}
 =\frac1{16\zeta(3)}.
\end{equation}
\end{corollary}

\begin{corollary}\label{cor:circuit}
Set
\begin{equation}\label{eq:epsilon-star}
 \epsilon_*=2\sin\frac{\pi}{2\sqrt3}.
\end{equation}
Fix $0\le\epsilon<\epsilon_*$.  For every sufficiently large power of two
$D=2^n$, with probability at least
\begin{equation}\label{eq:circuit-probability}
 1-2\exp\{-c_{\rm sh}M^{1/2}D^{7/4}\},
\end{equation}
every no-ancilla circuit built from arbitrary two-qubit gates and free
one-qubit gates whose endpoint $[V]$ satisfies
$\delta_{\op}([U],[V])\le\epsilon$ must contain at least
\begin{equation}\label{eq:circuit-main}
 \left(
 \frac1{\sqrt3}
 -\frac2\pi\arcsin\frac\epsilon2
 -C_{\rm sh}\eta_D
 \right)D-\sqrt n
\end{equation}
two-qubit gates.
\end{corollary}

\subsection*{Conventions}
All tangent spaces, determinants, and Jacobians in the geometric argument are
understood over the underlying real Hilbert spaces.  Complexification is used
only to diagonalize commutator operators; every unordered root pair
$\{i,j\}$ corresponds to a real two-plane and therefore contributes each root
singular value with multiplicity two.

Unless a subscript specifies otherwise, $\Prob$ and $\E$ refer to the
probability law stated in the surrounding argument.  We write
$X_D=O_{\Prob}(a_D)$ when $X_D/a_D$ is bounded in probability and
$X_D=o_{\Prob}(a_D)$ when $X_D/a_D\to0$ in probability.  The letters
$c,C>0$ denote constants that may change from line to line; subscripts record
permitted parameter dependence.  The usual meanings of $O$, $o$, $\Omega$,
$\Theta$, and $\asymp$ are used with the same convention.

We use $\mathbb F$ for either $\mathbb R$ or $\mathbb C$.  If the domain
of a linear map $L$ has dimension $k$, let
$s_1(L),\ldots,s_k(L)$ be its singular values, including zeros, and write
\[
 \|L\|_{\op}=\max_j s_j(L),\qquad
 \|L\|_{\HS}=\left(\sum_j s_j(L)^2\right)^{1/2},\qquad
 \|L\|_{\Sone}=\sum_j s_j(L).
\]
We also write $\sym L=(L+L^*)/2$.  A second subscript on an operator norm
indicates the Hilbert-space norm used in its definition; for example,
$\|T\|_{\op,0}$ is computed from $\|\cdot\|_0$.  For
$X,Y\in\mathfrak g$ and a group element $g$, our adjoint conventions are
\[
 \ad_X(Y)=[X,Y],\qquad \Ad_g(X)=gXg^{-1}.
\]
Let $\mathbb T=\{z\in\mathbb C:|z|=1\}$.  For Hermitian $H$,
$\diam\spec(H)$ denotes the largest eigenvalue minus the smallest, and
$\Arg z\in(-\pi,\pi]$ denotes the principal argument of $z\in\mathbb T$.
We write $\mathbf1_E$ for the indicator of an event or set $E$, and
$\zeta$ denotes the Riemann zeta function.  For vectors $x,y$ in a Hilbert
space, $x\otimes y^*$ denotes the rank-one operator
$v\mapsto x\langle y,v\rangle$.

If $E$ is a $k$-dimensional real or complex Hilbert space, $\mathcal H$ is
another Hilbert space, and $L:E\to\mathcal H$ is linear, write
\begin{equation}\label{eq:k-volume}
 \vol_k(L)=\det(L^*L)^{1/2}=\prod_{j=1}^k s_j(L).
\end{equation}
For a square map this is $|\det L|$.  For $\delta>0$ define the regularized
$k$-volume
\begin{equation}\label{eq:regularized-volume-definition}
 \vol_{k,\delta}(L)=\prod_{j=1}^k\max\{s_j(L),\delta\}.
\end{equation}
When $L$ is square we also write $\det_\delta L$ for this quantity.  At a
radial time $t=sD$ and for a relative cutoff $0<\rho\le1$, set
\begin{equation}\label{eq:relative-regularized-determinant}
 \det_{\rho,s}L:=\prod_{j=1}^k\max\{s_j(L),\rho s\}.
\end{equation}

\section{Geodesic equations and Jacobi fields}\label{sec:jacobi-dynamics}

\subsection{Conjugate times and spectral bounds}

We first prove the conjugate-time upper estimate of
Le Brigant--Lichtenfelz--Preston \cite[Corollary~3.5]{LBLP} in the present
right-invariant convention.  The argument uses the index form, as in
\cite[Lemma~3.1]{LBLP}; we derive that form directly from the energy.

\begin{proposition}\label{prop:conjugate-upper}
Let $g(t)$ be a nonconstant geodesic for the metric defined by $A$, starting
at the identity with velocity $u_0$.  Put $m_0=Au_0$ and
$\omega=\|\ad_{m_0}\|_{\op,0}$.  Then $\omega>0$, and the first conjugate
time $\tau_1$ satisfies
\begin{equation}\label{eq:conjugate-upper}
 \tau_1\le\frac{2\pi D^2}{\omega}.
\end{equation}
\end{proposition}

\begin{proof}
Calculations may be made in a lift of $g$ to $SU(D)$; the quotient map to
$\PU(D)$ is a local isometry for the lifted metric.  Write
$u=\dot g g^{-1}$ and $m=Au$.  For a smooth Lie-algebra-valued function
$v$ with $v(0)=v(T)=0$, consider the fixed-endpoint variation
\[
 g_\varepsilon(t)=g(t)e^{\varepsilon v(t)}.
\]
The right-trivialized velocity of this variation has the expansion
\begin{equation}\label{eq:variation-velocity-expansion}
 u_\varepsilon
 =u+\Ad_g\left(\varepsilon v'
       +\frac{\varepsilon^2}{2}[v,v']\right)+O(\varepsilon^3).
\end{equation}
Indeed,
\[
 (\partial_t e^{\varepsilon v})e^{-\varepsilon v}
 =\varepsilon\int_0^1 e^{r\varepsilon v}v'e^{-r\varepsilon v}\,dr,
\]
which gives \eqref{eq:variation-velocity-expansion} upon expansion.
The energy is
\[
 E(g_\varepsilon)=\frac12\int_0^T
 \langle u_\varepsilon,Au_\varepsilon\rangle_0\,dt.
\]
Its first variation vanishes because $g$ is a geodesic.  Since conjugation is
orthogonal for $\langle\cdot,\cdot\rangle_0$, this gives
\[
 0=\int_0^T\langle\Ad_{g^{-1}}m,v'\rangle_0\,dt
 \qquad\text{for every such }v.
\]
Integration by parts shows that $\Ad_{g^{-1}}m=m_0$ is constant.
Differentiating this identity also gives $m'=[u,m]$.

Let $Y=\partial_\varepsilon g_\varepsilon|_{\varepsilon=0}=gv$ be the
variation field, and denote its index form by $\mathcal I_g(Y,Y)$.
The second-variation formula and \eqref{eq:variation-velocity-expansion} give
\begin{align}
 \mathcal I_g(Y,Y)
 &=\left.\frac{d^2}{d\varepsilon^2}E(g_\varepsilon)
   \right|_{\varepsilon=0}\notag\\
 &=\int_0^T\left(
   \langle\Ad_gv',A\Ad_gv'\rangle_0
   +\langle m_0,[v,v']\rangle_0\right)dt.
 \label{eq:right-index-form}
\end{align}
Set $B=\ad_{m_0}$.  Ad-invariance of the reference inner product implies
$B^*=-B$ and
$\langle m_0,[v,v']\rangle_0=\langle Bv,v'\rangle_0$.
Since $A\le D^2I$, \eqref{eq:right-index-form} yields
\begin{equation}\label{eq:index-comparison}
 \mathcal I_g(Y,Y)
 \le\int_0^T\left(D^2\|v'\|_0^2+\langle Bv,v'\rangle_0\right)dt.
\end{equation}

The Lie algebra $\mathfrak{pu}(D)$ has zero center, so $m_0\ne0$ implies
$\omega>0$.  A real skew-adjoint operator has an orthogonal decomposition
into rotation planes and its kernel.  On a plane realizing its operator
norm, choose an orthonormal pair $e_1,e_2$ such that
\[
 Be_1=\omega e_2,\qquad Be_2=-\omega e_1.
\]
For $T>2\pi D^2/\omega$, take
\[
 f(t)=\sin\frac{\pi t}{T},\qquad
 w(t)=e^{-tB/(2D^2)}e_1,\qquad v(t)=f(t)w(t).
\]
This is a nonzero fixed-endpoint variation.  On the chosen plane,
$\|w\|_0=1$, $w'=-Bw/(2D^2)$, and $\|Bw\|_0=\omega$, so
\[
 \|v'\|_0^2=(f')^2+\frac{\omega^2}{4D^4}f^2,
 \qquad
 \langle Bv,v'\rangle_0=-\frac{\omega^2}{2D^2}f^2.
\]
Substitution in \eqref{eq:index-comparison} therefore gives
\begin{equation}\label{eq:index-test-bound}
 \mathcal I_g(Y,Y)
 \le D^2\int_0^T(f')^2\,dt
       -\frac{\omega^2}{4D^2}\int_0^T f^2\,dt
 =\frac{D^2\pi^2}{2T}-\frac{\omega^2T}{8D^2}<0.
\end{equation}
By the index theorem, a negative fixed-endpoint index form implies a
conjugate point in $(0,T)$ \cite{Chavel}.  Since this holds for every
$T>2\pi D^2/\omega$, the first conjugate time satisfies
\eqref{eq:conjugate-upper}.
\end{proof}

Inversion carries the right-invariant metric to the left-invariant metric
with the same inner product at the identity and changes $u_0$ to $-u_0$.
It preserves conjugate times, while
$\|\ad_{-Au_0}\|_{\op,0}=\|\ad_{Au_0}\|_{\op,0}$, in agreement with the
left-invariant statement in \cite{LBLP}.
For an initial velocity $u_0$ of $A$-norm one, let $c(u_0)$ be the cut time: the
supremum of the times up to which its geodesic minimizes distance from the
identity.  The standard cut-time inequality is $c(u_0)\le\tau_1$
\cite{Chavel}.

\begin{proposition}\label{prop:phase-range}
Let $u_0$ have $A$-norm one, $m_0=Au_0$, and
\begin{equation}\label{eq:H}
 H=-\frac{i}{D}m_0.
\end{equation}
If $s>0$ and $sD<c(u_0)$, then
\begin{equation}\label{eq:phase-range}
 s\,\diam\spec(H)<2\pi.
\end{equation}
\end{proposition}

\begin{proof}
If $h_1,\ldots,h_D$ are the eigenvalues of $H$, then, after diagonalization
and complexification,
\[
 \ad_{m_0}E_{ij}=iD(h_i-h_j)E_{ij}.
\]
On the underlying real root planes this is a rotation, and consequently
\[
 \|\ad_{m_0}\|_{\op,0}=D\diam\spec(H).
\]
Proposition~\ref{prop:conjugate-upper} and the cut-time inequality give
\[
 sD<c(u_0)\le\tau_1
 \le\frac{2\pi D^2}{D\diam\spec(H)}.
\]
Division by $D$ proves the claim.
\end{proof}

\subsection{The Euler--Jacobi equations}

Along a unit-speed geodesic $g(t)$ with $g(0)=I$, write the
right-trivialized velocity and momentum as
\[
 u=\dot g g^{-1},\qquad m=Au.
\]
For the right-invariant metric, the Euler equation is
\begin{equation}\label{eq:euler}
 m'=[u,m].
\end{equation}
Let $z=\delta g\,g^{-1}$ and $\eta=\delta m$.  Differentiating gives
\begin{equation}\label{eq:linear-body}
 z'=[u,z]+A^{-1}\eta,
 \qquad
 \eta'=[A^{-1}\eta,m]+[u,\eta].
\end{equation}
Let $O(t)=\Ad_{g(t)^{-1}}$.  Then $m_0=O(t)m(t)$ is constant.  Setting
$\widetilde z=Oz$ and $\widetilde\eta=O\eta$ in \eqref{eq:linear-body} yields
\begin{equation}\label{eq:spatial}
 \widetilde z'=G(t)\widetilde\eta,
 \qquad
 \widetilde\eta'=D_0G(t)\widetilde\eta,
 \qquad
 D_0=-\ad_{m_0},
\end{equation}
where
\begin{equation}\label{eq:movingG}
 G(t)=O(t)A^{-1}O(t)^{-1}=P(t)+qP(t)^\perp,
 \qquad P(t)=O(t)PO(t)^{-1}.
\end{equation}
We now drop tildes.

The moving low-weight projection satisfies the following trace-norm bounds.

\begin{lemma}\label{lem:P-bounds}
For every unit-speed geodesic,
\begin{equation}\label{eq:Pprime}
 \|P'(t)\|_{\Sone}\le4M,
 \qquad
 \|P(t)^\perp D_0P(t)\|_{\Sone}\le2MD.
\end{equation}
\end{lemma}

\begin{proof}
Let $e_1(0),\ldots,e_M(0)$ be the normalized Pauli basis of the low-weight
subspace and set
\[
 e_a(t)=O(t)e_a(0),\qquad
 P(t)=\sum_{a=1}^M e_a(t)\otimes e_a(t)^*.
\]
Conjugation preserves both the normalized Hilbert--Schmidt norm and the matrix
operator norm, so $\|e_a(t)\|_{\op}=1$.  If
$v(t)=\Ad_{g(t)^{-1}}u(t)$, then differentiating
$e_a(t)=g(t)^{-1}e_a(0)g(t)$ gives
\[
 e_a'(t)=[e_a(t),v(t)].
\]
Since $A\ge I$ and the geodesic has unit speed,
$\|v(t)\|_0=\|u(t)\|_0\le1$.  Hence
\begin{equation}\label{eq:moving-basis-column-bound}
 \sum_{a=1}^M\|e_a'(t)\|_0^2
 \le\sum_{a=1}^M\|[v(t),e_a(t)]\|_0^2
 \le4M.
\end{equation}
Differentiating $P(t)e_a(t)=e_a(t)$ shows
$P'(t)e_a(t)=P(t)^\perp e_a'(t)$.  Moreover, differentiating $P^2=P$ gives
$PP'P=P^\perp P'P^\perp=0$, so $P'$ is off diagonal relative to
$P\oplus P^\perp$.  Therefore
\[
 \|P'(t)\|_{\HS}^2
 =2\|P(t)^\perp P'(t)P(t)\|_{\HS}^2
 =2\sum_{a=1}^M\|P(t)^\perp e_a'(t)\|_0^2
 \le8M.
\]
The range of $P'$ is contained in the sum of the ranges of
$P^\perp P'P$ and its adjoint, so $\rank P'\le2M$.  The
rank--Hilbert--Schmidt inequality now gives
\[
 \|P'(t)\|_{\Sone}
 \le\sqrt{2M}\,\|P'(t)\|_{\HS}\le4M.
\]

For the second estimate, the columns of $P^\perp D_0P$ in the basis
$e_1(t),\ldots,e_M(t)$ satisfy
\[
 \sum_{a=1}^M\|P(t)^\perp D_0e_a(t)\|_0^2
 \le\sum_{a=1}^M\|[m_0,e_a(t)]\|_0^2
 \le4M\|m_0\|_0^2.
\]
Unit speed implies
$1=\langle u_0,Au_0\rangle_0$ and hence
$\|m_0\|_0^2=\|Au_0\|_0^2\le Q\langle u_0,Au_0\rangle_0=Q=D^2$.
Thus
\[
 \|P^\perp D_0P\|_{\HS}\le2D\sqrt M,
 \qquad \rank(P^\perp D_0P)\le M,
\]
and another application of the rank--Hilbert--Schmidt inequality yields
$\|P^\perp D_0P\|_{\Sone}\le2MD$.
\end{proof}

Let
\[
 T(t)=G(t)^{1/4}=P(t)+aP(t)^\perp,
 \qquad a=q^{1/4}=D^{-1/2},
\]
and rescale the Jacobi variables by
\begin{equation}\label{eq:quarter-vars}
 \zeta=T^{-1}z,
 \qquad
 \xi=T\eta.
\end{equation}
The augmented system is
\begin{equation}\label{eq:aug-actual}
 \frac d{dt}\binom\zeta\xi=
 \begin{pmatrix}
 -T^{-1}T'&G^{1/2}\\
 0&T'T^{-1}+TD_0GT^{-1}
 \end{pmatrix}
 \binom\zeta\xi.
\end{equation}

For a fixed $X\in S_0^{N-1}$, let the corresponding unit-speed geodesic
have initial velocity $u_0=A^{-1/2}X$.  The coefficients in
\eqref{eq:aug-actual} depend on this geodesic.  Given
$v\in\mathfrak g$, let $(\zeta_v,\xi_v)$ be the solution of
\eqref{eq:aug-actual} with
\begin{equation}\label{eq:actual-Jacobi-initial-data}
 \zeta_v(0)=0,\qquad \xi_v(0)=v,
\end{equation}
and define the Jacobi map in these variables by
\begin{equation}\label{eq:actual-Jacobi-definition}
 \widehat J(t,X)v:=\zeta_v(t).
\end{equation}
We suppress $X$ from the notation.

Relative to $P(t)\oplus P(t)^\perp$, set
$C_P(t)=P(t)^\perp P'(t)P(t)$ and write
\[
 D_0=\begin{pmatrix}D_{11}&D_{12}\\-D_{12}^*&D_{22}\end{pmatrix},
 \qquad
 P'=\begin{pmatrix}0&C_P^*\\C_P&0\end{pmatrix}.
\]
Direct multiplication gives
\begin{equation}\label{eq:quarter-block}
 TD_0GT^{-1}=
 \begin{pmatrix}
 D_{11}&a^3D_{12}\\-aD_{12}^*&qD_{22}
 \end{pmatrix},
 \qquad
 T'T^{-1}=
 \begin{pmatrix}
 0&(a^{-1}-1)C_P^*\\(1-a)C_P&0
 \end{pmatrix}.
\end{equation}
Let
\begin{equation}\label{eq:K}
 K=PD_0P+qP^\perp D_0P^\perp+[P',P].
\end{equation}
Then $K^*=-K$.  Compare \eqref{eq:aug-actual} with
\begin{equation}\label{eq:aug-ref}
 \mathcal K(t)=
 \begin{pmatrix}0&q^{1/2}I\\0&K(t)\end{pmatrix}.
\end{equation}

\begin{proposition}\label{prop:quarter-action}
Let $\mathcal E$ be the difference between the generators in
\eqref{eq:aug-actual} and \eqref{eq:aug-ref}.  Then
\begin{equation}\label{eq:Epoint}
 \|\mathcal E(t)\|_{\Sone}
 \le C\left[
 q^{1/4}\|P^\perp D_0P\|_{\Sone}
 +q^{-1/4}\|P'\|_{\Sone}+M
 \right].
\end{equation}
Consequently, for every fixed $x>0$, uniformly for $0<s\le x$,
\begin{equation}\label{eq:Eint}
 \int_0^{sD}\|\mathcal E(t)\|_{\Sone}\,dt
 \le C_xMD^{3/2}s.
\end{equation}
\end{proposition}

\begin{proof}
Write $a=q^{1/4}$.  Since $T=aI+(1-a)P$, one has
\[
 T^{-1}T'=
 \begin{pmatrix}
 0&(1-a)C_P^*\\
 (a^{-1}-1)C_P&0
 \end{pmatrix},
\]
and hence
\begin{equation}\label{eq:quarter-upper-left-bound}
 \|T^{-1}T'\|_{\Sone}
 \le C a^{-1}\|P'\|_{\Sone}.
\end{equation}
The upper-right difference between the actual and reference generators is
\[
 G^{1/2}-q^{1/2}I=(1-q^{1/2})P,
\]
which has rank $M$ and trace norm at most $M$.

It remains to estimate the lower-right block.  Relative to
$P\oplus P^\perp$,
\[
 [P',P]=
 \begin{pmatrix}0&-C_P^*\\C_P&0\end{pmatrix},
 \qquad
 K=\begin{pmatrix}D_{11}&-C_P^*\\C_P&qD_{22}\end{pmatrix}.
\]
Subtracting $K$ from the lower-right block in \eqref{eq:aug-actual} and using
\eqref{eq:quarter-block} gives the exact error
\begin{equation}\label{eq:quarter-lower-right-error}
 \mathcal E_{22}=
 \begin{pmatrix}
 0&a^3D_{12}+a^{-1}C_P^*\\
 -aD_{12}^*-aC_P&0
 \end{pmatrix}.
\end{equation}
Because $0<a\le1$, the ideal property of the trace norm implies
\begin{equation}\label{eq:quarter-lower-right-bound}
 \|\mathcal E_{22}\|_{\Sone}
 \le C\left(
 a\|P^\perp D_0P\|_{\Sone}
 +a^{-1}\|P'\|_{\Sone}
 \right).
\end{equation}
Combining \eqref{eq:quarter-upper-left-bound},
\eqref{eq:quarter-lower-right-bound}, and the upper-right rank bound proves
\eqref{eq:Epoint}.

Finally, Lemma~\ref{lem:P-bounds} and $a=D^{-1/2}$ give the pointwise estimate
\[
 \|\mathcal E(t)\|_{\Sone}
 \le C\bigl(MD^{1/2}+M\bigr).
\]
Integration over $0\le t\le sD$ yields
\[
 \int_0^{sD}\|\mathcal E(t)\|_{\Sone}\,dt
 \le C_xMD^{3/2}s,
\]
which is \eqref{eq:Eint}.
\end{proof}

The propagator of \eqref{eq:aug-ref} is
\begin{equation}\label{eq:Mref}
 \mathcal M_K(t)=
 \begin{pmatrix}
 I&q^{1/2}\int_0^tW(r)\,dr\\
 0&W(t)
 \end{pmatrix},
\end{equation}
where $W'=KW$, $W(0)=I$; since $K^*=-K$, the propagator $W$ is unitary.
For $t\le xD$, both $\mathcal M_K(t)$ and its
inverse are bounded by a constant depending only on $x$.

\subsection{Comparison with a constant-coefficient system}

Let $R_K$ solve
\begin{equation}\label{eq:Kato-transport}
 R_K'=[P',P]R_K,\qquad R_K(0)=I.
\end{equation}
The generator $[P',P]$ is skew-adjoint, so $R_K$ is unitary.  Moreover,
using $P'=PP'+P'P$ and $P P'P=0$,
\[
 \frac d{dt}(R_K^*P(t)R_K)
 =R_K^*\bigl(P'+[P,[P',P]]\bigr)R_K=0.
\]
Consequently,
\begin{equation}\label{eq:Kato-fixed-projection}
 R_K(t)^*P(t)R_K(t)=P_0:=P(0).
\end{equation}
This is the standard Kato transport of the moving low-weight subspace
\cite{Kato}.

With $W$ as above, put
\[
 \widehat W=R_K^*W,\qquad
 \widehat D=R_K^*D_0R_K.
\]
Equations \eqref{eq:K} and \eqref{eq:Kato-fixed-projection} give
\begin{equation}\label{eq:Khat}
 \widehat W'=\widehat K\widehat W,
 \qquad
 \widehat K=P_0\widehat DP_0
 +qP_0^\perp\widehat DP_0^\perp.
\end{equation}
Let $\overline W(0)=I$ solve
\begin{equation}\label{eq:Kbar}
 \overline W'=\overline K\overline W,
 \qquad
 \overline K=qP_0\widehat DP_0
 +qP_0^\perp\widehat DP_0^\perp.
\end{equation}
Both generators are block diagonal for the fixed splitting
$P_0\mathfrak g\oplus P_0^\perp\mathfrak g$, and their restrictions to
$P_0^\perp\mathfrak g$ coincide.  Uniqueness for the two block equations
therefore gives
\begin{equation}\label{eq:cheap-block-rank}
 \rank(\widehat W(r)-\overline W(r))\le M,
 \qquad
 \|\widehat W(r)-\overline W(r)\|_{\Sone}\le2M.
\end{equation}

The propagator $e^{qD_0r}$, expressed in the same frame, is
\[
 \widehat U(r)=R_K(r)^*e^{qD_0r}.
\]
Its generator is
\[
 q\widehat D-R_K^*[P',P]R_K.
\]
The off-diagonal blocks of $q\widehat D$ are absent from $\overline K$, and
unitary invariance of the trace norm gives
\begin{equation}\label{eq:phase-generator-tail}
 \left\|\overline K-
 \bigl(q\widehat D-R_K^*[P',P]R_K\bigr)\right\|_{\Sone}
 \le2q\|P^\perp D_0P\|_{\Sone}+\|P'\|_{\Sone}.
\end{equation}
Since both propagators are unitary, Duhamel's formula and
Lemma~\ref{lem:P-bounds} imply, at every time $r$,
\begin{equation}\label{eq:Kbar-frozen-Duhamel}
 \|\overline W(r)-\widehat U(r)\|_{\Sone}
 \le CM\left(r+\frac rD\right).
\end{equation}
Combining \eqref{eq:cheap-block-rank} and
\eqref{eq:Kbar-frozen-Duhamel}, and conjugating back by $R_K(r)$, yields
\begin{equation}\label{eq:W-frozen-prefix}
 \|W(r)-e^{qD_0r}\|_{\Sone}
 \le2M+CM\left(r+\frac rD\right).
\end{equation}

\begin{proposition}\label{prop:frozen-tail}
Let
\[
 \widehat J_K(t)=q^{1/2}\int_0^tW(r)\,dr,
 \qquad
 \widehat J_0(t)=q^{1/2}\int_0^te^{qD_0r}\,dr.
\]
For every fixed $x>0$, uniformly for $0<s\le x$,
\begin{equation}\label{eq:Jtail}
 \|\widehat J_K(sD)-\widehat J_0(sD)\|_{\Sone}
 \le C_x\bigl(Ms+MDs^2\bigr).
\end{equation}
\end{proposition}

\begin{proof}
Since $q^{1/2}=D^{-1}$, integration of
\eqref{eq:W-frozen-prefix} gives
\begin{align*}
 \|\widehat J_K(sD)-\widehat J_0(sD)\|_{\Sone}
 &\le \frac1D\int_0^{sD}
 \left[2M+CM\left(r+\frac rD\right)\right]dr\\
 &\le C_x\bigl(Ms+MDs^2\bigr).
\end{align*}
\end{proof}

\section{Jacobian estimates and polar integration}\label{sec:regularized-polar}

\subsection{A determinant comparison}

We now pass from estimates for the augmented propagators to estimates for
the determinants of their Jacobi blocks.  Regularizing the singular values
allows the comparison block to be singular.  We first record an estimate
for isometric embeddings.

\begin{lemma}\label{lem:principal}
Let $Q_1,Q_0:\mathbb F^k\to\mathcal H$ be isometric embeddings, and let
$\Pi_1,
\Pi_0$ be the orthogonal projections onto their ranges.  There is an
orthogonal or unitary map $V$ on $\mathbb F^k$ such that
\begin{equation}\label{eq:principal}
 \|Q_1V-Q_0\|_{\Sone}\le\|\Pi_1-\Pi_0\|_{\Sone}.
\end{equation}
\end{lemma}

\begin{proof}
Choose principal-vector bases for the two ranges, with principal angles
$\theta_j\in[0,\pi/2]$, and let $V$ identify the corresponding domain
vectors.  The singular values of $Q_1V-Q_0$ are
$2\sin(\theta_j/2)$.  The nonzero singular values of
$\Pi_1-\Pi_0$ are $\sin\theta_j$, each with multiplicity two.  Therefore
\[
 \|Q_1V-Q_0\|_{\Sone}
 =2\sum_j\sin(\theta_j/2)
 \le2\sum_j\sin\theta_j
 =\|\Pi_1-\Pi_0\|_{\Sone}.
\]
\end{proof}

\begin{theorem}\label{thm:output}
Fix $T>0$ and let $\mathcal H$ be a finite-dimensional real or complex
Hilbert space.  Let $\mathcal M(t)=\mathcal M_0(t)R(t)$ on
$0\le t\le T$, where
\begin{equation}\label{eq:R-equation}
 R'=\widetilde E R,\qquad R(0)=I,
\end{equation}
and suppose that, for some $K_0<\infty$,
\begin{equation}\label{eq:reference-bounds}
 \sup_{0\le t\le T}
 \bigl(\|\mathcal M_0(t)\|_{\op}
       +\|\mathcal M_0(t)^{-1}\|_{\op}\bigr)\le K_0.
\end{equation}
Let $\iota:\mathbb F^k\to\mathcal H$ be an isometric embedding and
$C:\mathcal H\to\mathbb F^k$ a bounded linear map.  Put
\begin{equation}\label{eq:three-output-blocks}
 B=C\mathcal M(T)\iota,
 \qquad B_K=C\mathcal M_0(T)\iota,
\end{equation}
and let $B_0:\mathbb F^k\to\mathbb F^k$ satisfy, for some $\ell\ge0$,
\begin{equation}\label{eq:BK-B0}
 \|B_K-B_0\|_{\Sone}\le\ell.
\end{equation}
Define
\begin{equation}\label{eq:eta-definitions}
 \eta=\int_0^T\|\widetilde E(t)\|_{\Sone}\,dt,
 \qquad
 \eta_{\rm sym}=\int_0^T
   \|\sym\widetilde E(t)\|_{\Sone}\,dt.
\end{equation}
Then for every $\delta>0$,
\begin{equation}\label{eq:output}
 \log\vol_k(B)
 \le
 \log\vol_{k,\delta}(B_0)
 +\eta_{\rm sym}
 +C_0\delta^{-1}(\eta+\ell),
\end{equation}
where $C_0$ depends only on $K_0$ and $\|C\|_{\op}$.
\end{theorem}

\begin{proof}
Because $R(t)$ is invertible, $Y(t):=R(t)\iota$ has full column rank.  Write
its thin polar decomposition as
\begin{equation}\label{eq:thin-polar-factorization}
 Y=\mathcal QH,
 \qquad \mathcal Q^*\mathcal Q=I_k,
 \qquad H=(Y^*Y)^{1/2}>0,
\end{equation}
and put $\Pi=\mathcal Q\mathcal Q^*$.  We first record the two differential identities used in
the estimate.

Since $Y'=\widetilde EY$, differentiation of $Y^*Y=H^2$ gives
\[
 \frac d{dt}\log\det H
 =\frac12\Tr\bigl((Y^*Y)^{-1}(Y^*Y)'\bigr)
 =\Re\Tr(\mathcal Q^*\widetilde E\mathcal Q)
 =\Tr(\Pi\sym\widetilde E).
\]
Also, using the identity
$\Pi=Y(Y^*Y)^{-1}Y^*$ and differentiating it directly yields
\begin{equation}\label{eq:range-projection-derivative}
 \Pi'=(I-\Pi)\widetilde E\Pi
       +\Pi\widetilde E^*(I-\Pi).
\end{equation}
Consequently,
\begin{equation}\label{eq:polar-integrated-bounds}
 \log\det H(T)\le\eta_{\rm sym},
 \qquad
 \|\Pi(T)-\Pi(0)\|_{\Sone}\le2\eta.
\end{equation}
Here we used the ideal property of the trace norm for the two off-diagonal
compressions in \eqref{eq:range-projection-derivative}.

At $t=0$ one has $Y(0)=\iota$, so $\Pi(0)=\iota\iota^*$.  By
Lemma~\ref{lem:principal}, there is an orthogonal or unitary $V$ on the input
space such that
\begin{equation}\label{eq:Q-alignment}
 \|\mathcal Q(T)V-\iota\|_{\Sone}\le2\eta.
\end{equation}
Right multiplication by $V$ does not change $k$-volume.  Using
$\mathcal M=\mathcal M_0\mathcal QH$ on the input columns, we obtain
\begin{align}
 \vol_k(B)
 &=\det H(T)\,
   \vol_k\bigl(C\mathcal M_0(T)\mathcal Q(T)V\bigr).
 \label{eq:output-volume-factorization}
\end{align}
Set $L_T=C\mathcal M_0(T)\mathcal Q(T)V$.  By
\eqref{eq:reference-bounds}, \eqref{eq:Q-alignment}, and
\eqref{eq:BK-B0},
\begin{equation}\label{eq:LT-B0-bound}
 \|L_T-B_0\|_{\Sone}
 \le 2\|C\|_{\op}K_0\eta+\ell.
\end{equation}
Mirsky's inequality \cite{Bhatia} gives
\[
 \sum_{j=1}^k|s_j(L_T)-s_j(B_0)|
 \le\|L_T-B_0\|_{\Sone}.
\]
The scalar function $r\mapsto\log\max\{r,\delta\}$ is
$\delta^{-1}$-Lipschitz on $[0,\infty)$.  Therefore
\begin{align*}
 \log\vol_k(L_T)
 &\le\sum_{j=1}^k\log\max\{s_j(L_T),\delta\}\\
 &\le\log\vol_{k,\delta}(B_0)
   +\delta^{-1}\|L_T-B_0\|_{\Sone}.
\end{align*}
Combining this with \eqref{eq:polar-integrated-bounds},
\eqref{eq:output-volume-factorization}, and
\eqref{eq:LT-B0-bound} proves \eqref{eq:output}.
\end{proof}

For the Jacobi system, apply Theorem~\ref{thm:output} with the following
input and output maps.  Let $\mathcal M(t)$ be the propagator of
\eqref{eq:aug-actual}, let $\mathcal M_K(t)$ be the propagator
\eqref{eq:Mref}, and put
\begin{equation}\label{eq:relative-propagator}
 R(t)=\mathcal M_K(t)^{-1}\mathcal M(t).
\end{equation}
Then
\begin{equation}\label{eq:Etilde-application}
 R'=\widetilde E R,
 \qquad
 \widetilde E=\mathcal M_K^{-1}\mathcal E\mathcal M_K.
\end{equation}
For $t\le xD$, \eqref{eq:Mref} and the unitarity of $W$ imply
\begin{equation}\label{eq:MK-uniform-bound}
 \|\mathcal M_K(t)\|_{\op}
 +\|\mathcal M_K(t)^{-1}\|_{\op}\le C_x.
\end{equation}
Indeed,
$q^{1/2}\|\int_0^tW(r)\,dr\|_{\op}\le q^{1/2}t\le x$.

Take the augmented input and output maps
\begin{equation}\label{eq:input-output-identification}
 \iota(v)=\binom{0}{v},
 \qquad
 C\binom{\zeta}{\xi}=\zeta.
\end{equation}
By construction,
\begin{equation}\label{eq:three-Jacobi-blocks}
 C\mathcal M(sD)\iota=\widehat J(sD),
 \quad
 C\mathcal M_K(sD)\iota=\widehat J_K(sD),
 \quad
 B_0=\widehat J_0(sD).
\end{equation}
Thus Proposition~\ref{prop:frozen-tail} supplies the quantity $\ell$ in
\eqref{eq:BK-B0}.

\begin{proposition}\label{prop:actual-frozen}
For every fixed $x>0$, uniformly for $0<s\le x$ and $0<\rho\le1$,
\begin{equation}\label{eq:actual-frozen}
 \log|\det\widehat J(sD)|
 \le
 \log\det_{\rho,s}\widehat J_0(sD)
 +C_x\left(\rho^{-1}MD^{3/2}+MD^{3/2}\right).
\end{equation}
\end{proposition}

\begin{proof}
The ideal property of $\|\cdot\|_{\Sone}$ and
\eqref{eq:MK-uniform-bound} give
\[
 \int_0^{sD}\|\widetilde E(t)\|_{\Sone}\,dt
 \le C_x\int_0^{sD}\|\mathcal E(t)\|_{\Sone}\,dt
 \le C_xMD^{3/2}s
\]
by Proposition~\ref{prop:quarter-action}; the same bound applies to the
symmetric part.  Proposition~\ref{prop:frozen-tail} gives
\[
 \ell\le C_x(Ms+MDs^2).
\]
Apply Theorem~\ref{thm:output} with $T=sD$ and
$\delta=\rho s$.  Since both $\eta$ and $\ell$ vanish at least linearly in
$s$,
\[
 (\rho s)^{-1}(\eta+\ell)
 \le C_x\rho^{-1}MD^{3/2}
\]
uniformly down to $s=0$.  The remaining term
$\eta_{\rm sym}$ is at most $C_xMD^{3/2}$, which proves
\eqref{eq:actual-frozen}.
\end{proof}

\subsection{The polar Jacobian}

To obtain the polar Jacobian from $\widehat J$, we undo the changes of
variables at the initial and final times and divide by the radial singular
value.

For $X\in S_0^{N-1}$, put
\[
 u_0=A^{-1/2}X,
\]
so that $u_0$ has $A$-norm one.  Write
\[
 \Phi(t,X)=\operatorname{Exp}_{[I]}(tA^{-1/2}X)
\]
for the Riemannian exponential map in these Euclidean coordinates.

\begin{proposition}\label{prop:polar-jacobian}
Let $t>0$ and let $\widehat J(t,X)$ be the square map from the initial
rescaled momentum
$\xi_0$ to the rescaled position $\zeta(t)$ in
\eqref{eq:quarter-vars}.  The metric-normalized differential of
$X\mapsto\Phi(t,X)$ on the full $N$-dimensional initial space is
\begin{equation}\label{eq:full-J}
 \mathcal J(t,X)
 =G(t)^{-1/4}\widehat J(t,X)G(0)^{-1/4}.
\end{equation}
Consequently,
\begin{equation}\label{eq:detfull}
 |\det\mathcal J(t,X)|
 =Q^{(N-M)/2}|\det\widehat J(t,X)|.
\end{equation}
If $j_A(t,X)$ denotes the normal polar Jacobian, then
\begin{equation}\label{eq:jq}
 j_A(t,X)
 =Q^{(N-M)/2}\frac{|\det\widehat J(t,X)|}{t}.
\end{equation}
\end{proposition}

\begin{proof}
Let $h\in\mathfrak g$ be a variation of the Euclidean coordinate $X$.  The
initial velocity variation is $A^{-1/2}h$, so the initial momentum variation
is
\[
 \eta_0=A(A^{-1/2}h)=A^{1/2}h=G(0)^{-1/2}h.
\]
Since $\xi=T\eta$ and $T=G^{1/4}$,
\begin{equation}\label{eq:initial-quarter-normalization}
 \xi_0=G(0)^{-1/4}h.
\end{equation}
At the endpoint, the spatial position variation is
$z=G(t)^{1/4}\zeta$.  The right-invariant metric, expressed in the spatial
frame, has inverse inertia $G(t)$; hence the metric-normalized endpoint
coordinate is
\[
 G(t)^{-1/2}z=G(t)^{-1/4}\zeta.
\]
Combining this identity with \eqref{eq:initial-quarter-normalization} gives
\eqref{eq:full-J}.

The projection $P(t)$ is an orthogonal conjugate of $P(0)$, so
\[
 \det G(t)=\det G(0)=q^{N-M}=Q^{-(N-M)}.
\]
Each endpoint factor $G^{-1/4}$ therefore has determinant
$Q^{(N-M)/4}$, which proves \eqref{eq:detfull}.

It remains to remove the radial column.  Varying $X$ in the radial direction
$h=X$ gives
\[
 \Phi(t,(1+\varepsilon)X)
 =\gamma_X((1+\varepsilon)t),
\]
where $\gamma_X$ is the unit-speed geodesic with initial velocity
$A^{-1/2}X$.  Thus the radial Jacobi field is $t\dot\gamma_X(t)$ and has
metric norm $t$.  By the Gauss lemma it is orthogonal to the images of the
angular variations $h\perp X$.  Hence the full $N$-volume is $t$ times the
normal $(N-1)$-volume, proving \eqref{eq:jq}.
\end{proof}

Let $c(X):=c(A^{-1/2}X)$ be the cut time of the corresponding unit-speed
geodesic.  Let $d\varsigma_0$ be the surface measure on $S_0^{N-1}$ induced by
$\langle\cdot,\cdot\rangle_0$, and let
$\omega_N=\pi^{N/2}/\Gamma(N/2+1)$ be the volume of the Euclidean unit ball
in $\mathbb R^N$.  Thus $\varsigma_0(S_0^{N-1})=N\omega_N$.  In the following polar formulas,
$\E$ denotes expectation with respect to normalized surface probability on
$S_0^{N-1}$.  For $R>0$, the standard polar integration formula on the minimizing domain
of the exponential map \cite{Chavel} gives
\begin{equation}\label{eq:polar-before-scaling}
 \Vol_A B_A([I],R)
 =\int_{S_0^{N-1}}
   \int_0^R\mathbf1_{\{t<c(X)\}}j_A(t,X)\,dt\,d\varsigma_0(X).
\end{equation}
The cut locus has Riemannian measure zero, so no additional boundary term is
present.  Substituting \eqref{eq:jq} and setting $R=xD$ and $t=sD$ yields
\begin{equation}\label{eq:polar-unscaled-volume}
 \Vol_A B_A([I],xD)
 =N\omega_NQ^{(N-M)/2}
 \int_0^x\E\left[
  \mathbf1_{\{sD<c(X)\}}|\det\widehat J(sD,X)|
 \right]\frac{ds}{s}.
\end{equation}
On the other hand, the metric tensor at the identity is obtained from
$\langle\cdot,\cdot\rangle_0$ by the positive operator $A$, so
\begin{equation}\label{eq:total-volume-scaling}
 \Vol_A\PU(D)
 =\det(A)^{1/2}\Vol_0\PU(D)
 =Q^{(N-M)/2}\Vol_0\PU(D).
\end{equation}
The inertia determinant cancels exactly, and therefore
\begin{equation}\label{eq:polar-normalized}
 \begin{aligned}
 \mu_D\!\left(B_A([I],xD)\right)
 &=\frac{\Vol_A B_A([I],xD)}{\Vol_A\PU(D)}\\
 &=\frac{N\omega_N}{\Vol_0\PU(D)}
 \int_0^x\E\left[
  \mathbf1_{\{sD<c(X)\}}|\det\widehat J(sD,X)|
 \right]\frac{ds}{s}.
 \end{aligned}
\end{equation}
A right-invariant Riemannian volume on the compact group $\PU(D)$ is a
constant multiple of right Haar measure.  After division by total volume,
\eqref{eq:polar-normalized} is therefore a Haar-probability identity.

The singular values of $\widehat J_0$, counted with their real
multiplicities, are as follows.  Let
\[
 H=-\frac{i}{D}m_0
\]
have eigenvalues $h_1,\ldots,h_D$.  On the $(D-1)$-dimensional real Cartan
subspace, $D_0=-\ad_{m_0}$ vanishes, so
$\widehat J_0(sD)=q^{1/2}sD\,I=sI$.  For every $i<j$, the corresponding
real root plane is invariant under $D_0$ and $D_0$ acts there as a rotation
with angular frequency $D(h_i-h_j)$.  Equivalently, on a complex root vector
with $D_0v=i\omega v$,
\begin{equation}\label{eq:frozen-spectral-output}
 \widehat J_0(sD)v
 =q^{1/2}\frac{e^{iq\omega sD}-1}{iq\omega}v.
\end{equation}
It follows that the root singular value is
\begin{equation}\label{eq:root-sing}
 s\left|\sinc\frac{s(h_i-h_j)}2\right|,
\end{equation}
with real multiplicity two.  In particular,
\begin{equation}\label{eq:regularized-frozen-product}
 \det_{\rho,s}\widehat J_0(sD)
 =s^{D-1}\prod_{i<j}
 \left[
  s\max\left\{
   \left|\sinc\frac{s(h_i-h_j)}2\right|,\rho
  \right\}
 \right]^2.
\end{equation}
Thus the total real multiplicity is
$(D-1)+2\binom D2=D^2-1=N$.

\section{Spherical averages and Weyl integration}\label{sec:abel-weyl}

\subsection{Comparison of spherical averages}

We compare the average over the unit sphere in the high-weight subspace
with an average over the full trace-zero Hermitian sphere, whose measure
is invariant under unitary conjugation.  The comparison uses nonnegative
integrands and includes an additive error from an exceptional Gaussian event.

Write $\mathfrak g=\mathfrak g_{\cL}\oplus\mathfrak g_{\cP}$ for the
low-weight and high-weight subspaces, of dimensions $M$ and $N-M$,
respectively, and set
\[
 \mathfrak h=-i\mathfrak g,\qquad
 \mathfrak h_{\cL}=-i\mathfrak g_{\cL},\qquad
 \mathfrak h_{\cP}=-i\mathfrak g_{\cP}.
\]
Thus $\mathfrak h$ is the real Hilbert space of trace-zero Hermitian matrices
with the normalized Hilbert--Schmidt norm.  The metric-sphere coordinate
$X\in S_0^{N-1}$ decomposes as
\begin{equation}\label{eq:metric-sphere-split}
 X=\sqrt z\,X_{\cL}+\sqrt{1-z}\,X_{\cP},
\end{equation}
where $z,X_{\cL},X_{\cP}$ are independent,
\begin{equation}\label{eq:beta}
 z\sim\operatorname{Beta}\left(\frac M2,\frac{N-M}{2}\right),
\end{equation}
and $X_{\cL},X_{\cP}$ are uniform on the unit spheres of
$\mathfrak g_{\cL}$ and $\mathfrak g_{\cP}$.  Here
$\operatorname{Beta}(a,b)$ has density
$z^{a-1}(1-z)^{b-1}/B(a,b)$ on $(0,1)$, where
$B(a,b)=\Gamma(a)\Gamma(b)/\Gamma(a+b)$.  If
$H_{\cL}=-iX_{\cL}$ and $H_{\cP}=-iX_{\cP}$, then the normalized momentum
from \eqref{eq:H} is
\begin{equation}\label{eq:Hsplit}
 H=\sqrt{1-z}\,H_{\cP}+\frac{\sqrt z}{D}H_{\cL}.
\end{equation}
Every low-weight Hilbert--Schmidt unit vector satisfies
\begin{equation}\label{eq:cheap-op-bound}
 \|H_{\cL}\|_{\op}\le\sqrt M.
\end{equation}
Indeed, expand it in an orthonormal low-weight Pauli basis, use that every basis
matrix has operator norm one, and apply Cauchy--Schwarz to the coefficients.

Set $x_\star=\pi/\sqrt3$.  For a trace-zero Hermitian matrix $K$ with eigenvalues
$\kappa_1,\ldots,\kappa_D$, define
\begin{equation}\label{eq:root-functional-general}
 \mathcal R_{s,\rho}(K)
 :=s^{D-1}\prod_{i<j}
 \left[
  s\max\left\{
   \left|\sinc\frac{s(\kappa_i-\kappa_j)}2\right|,\rho
  \right\}
 \right]^2.
\end{equation}
For a unit vector $K\in\mathfrak h$ and $0<y\le s$, write
\begin{equation}\label{eq:root-functional-effective-scale}
 \mathcal R_{s,y,\rho}(K)
 :=s^{D-1}\prod_{i<j}
 \left[
  s\max\left\{
   \left|\sinc\frac{y(\kappa_i-\kappa_j)}2\right|,\rho
  \right\}
 \right]^2.
\end{equation}
Thus
$\mathcal R_{s,y,\rho}(K)=\mathcal R_{s,\rho}((y/s)K)$.
Finally, for $L>0$ put
\begin{equation}\label{eq:F-transfer-definition}
 F_{s,y,\rho,L}(K)
 =\mathbf1_{\{y\diam\spec(K)\le L\}}
  \mathcal R_{s,y,\rho}(K).
\end{equation}
This is a nonnegative function on the unit sphere.

The logarithm of $\mathcal R_{s,y,\rho}$ satisfies the following
perturbation bound.

\begin{lemma}\label{lem:root-lipschitz}
There is an absolute constant $C$ such that the following holds.  Let
$0<y\le s\le x_\star$, $0<\rho\le1$, $L>0$, and $\delta\ge0$, and
let $K,K'$ be trace-zero Hermitian unit vectors satisfying
$\|K-K'\|_{\op}\le\delta$.  Then
\begin{equation}\label{eq:root-log-lipschitz}
 \left|
  \log\mathcal R_{s,y,\rho}(K)
  -\log\mathcal R_{s,y,\rho}(K')
 \right|
 \le Cx_\star\rho^{-1}D^2\delta.
\end{equation}
Moreover,
\begin{equation}\label{eq:diameter-enlargement}
 y\diam\spec(K)\le L
 \quad\Longrightarrow\quad
 y\diam\spec(K')\le L+2x_\star\delta.
\end{equation}
\end{lemma}

\begin{proof}
The scalar function
\[
 g_\rho(t)=\max\{|\sinc(t/2)|,\rho\}
\]
is globally Lipschitz with an absolute Lipschitz constant, and is bounded
below by $\rho$.  Hence $\log g_\rho$ is $C\rho^{-1}$-Lipschitz.  Write $\kappa_1\le\cdots\le\kappa_D$ and
$\kappa_1'\le\cdots\le\kappa_D'$ for the ordered eigenvalues of $K$ and
$K'$.  Weyl's eigenvalue perturbation inequality gives
$|\kappa_i-\kappa_i'|\le\delta$, so every eigenvalue difference changes by at
most $2\delta$.  Summing the corresponding log changes over
$2\binom D2$ real root factors proves \eqref{eq:root-log-lipschitz}.
The same eigenvalue estimate gives
$|\diam\spec(K)-\diam\spec(K')|\le2\delta$, proving
\eqref{eq:diameter-enlargement}.
\end{proof}

To compare the spherical averages, we add a small Gaussian component in
the low-weight subspace.

Let $X$ and $Z$ be independent standard Gaussians in
$\mathfrak h_{\cP}$ and $\mathfrak h_{\cL}$, respectively; here standard
means covariance equal to the identity for the normalized Hilbert--Schmidt
inner product.  Put
\begin{equation}\label{eq:Ytau-definition}
 \tau=e^{-\sqrt D},
 \qquad
 Y_\tau=X+\tau Z,
 \qquad
 Y_1=X+Z.
\end{equation}
For a nonzero vector $V$, write $\widehat V=V/\|V\|_0$.

\begin{lemma}\label{lem:density-domination}
For every nonnegative measurable function $F$ on the full unit sphere,
\begin{equation}\label{eq:density-domination}
 \E F(\widehat Y_\tau)
 \le\tau^{-M}\E F(\widehat Y_1)
 =e^{M\sqrt D}\E F(\widehat Y_1).
\end{equation}
\end{lemma}

\begin{proof}
With respect to the orthogonal splitting
$\mathfrak h_{\cP}\oplus\mathfrak h_{\cL}$, the density of $Y_\tau$ is
\[
 f_\tau(y)
 =(2\pi)^{-N/2}\tau^{-M}
 \exp\left(-\frac12\|y_{\cP}\|_0^2
           -\frac1{2\tau^2}\|y_{\cL}\|_0^2\right).
\]
Relative to the standard full Gaussian density $f_1$ of $Y_1$,
\[
 \frac{f_\tau(y)}{f_1(y)}
 =\tau^{-M}
 \exp\left[-\frac12(\tau^{-2}-1)\|y_{\cL}\|_0^2\right]
 \le\tau^{-M}.
\]
Apply this pointwise density bound to the degree-zero extension
$y\mapsto F(y/\|y\|_0)$.
\end{proof}

\begin{lemma}\label{lem:direction-perturbation}
Let $K_D=N-M$ and set
\begin{equation}\label{eq:good-event-parameters}
 a_D=e^{-\sqrt D/4},
 \qquad
 b_D=e^{\sqrt D/2},
 \qquad
 r_D=\frac{\tau b_D}{a_D\sqrt{K_D}},
\end{equation}
\begin{equation}\label{eq:deltaD-explicit}
 \delta_D=\sqrt M\,r_D+\frac12\sqrt D\,r_D^2.
\end{equation}
On the event
\begin{equation}\label{eq:good-Gaussian-event}
 \mathcal G_D=
 \left\{\|X\|_0\ge a_D\sqrt{K_D},\quad
        \|Z\|_0\le b_D\right\},
\end{equation}
one has
\begin{equation}\label{eq:direction-op-distance}
 \|\widehat X-\widehat Y_\tau\|_{\op}\le\delta_D,
 \qquad
 \delta_D=e^{-\Omega(\sqrt D)}.
\end{equation}
Moreover, for all sufficiently large $D$,
\begin{equation}\label{eq:bad-Gaussian-probability}
 \Prob(\mathcal G_D^c)
 \le e^{-cD^2\sqrt D}+e^{-ce^{\sqrt D}}
\end{equation}
for an absolute $c>0$.
\end{lemma}

\begin{proof}
Because $X$ and $Z$ lie in orthogonal subspaces, if
$r=\tau\|Z\|_0/\|X\|_0$, then
\[
 \widehat Y_\tau
 =\frac1{\sqrt{1+r^2}}\widehat X
  +\frac r{\sqrt{1+r^2}}\widehat Z.
\]
Every full Hilbert--Schmidt unit vector has operator norm at most $\sqrt D$,
and \eqref{eq:cheap-op-bound} gives
$\|\widehat Z\|_{\op}\le\sqrt M$.  Since
$|1-(1+r^2)^{-1/2}|\le r^2/2$, on $\mathcal G_D$ we have
\[
 \|\widehat X-\widehat Y_\tau\|_{\op}
 \le r\sqrt M+\frac12r^2\sqrt D
 \le\delta_D.
\]
The displayed expression for $r_D$ and $K_D\asymp D^2$ show that
$\delta_D=e^{-\Omega(\sqrt D)}$.

For a standard Gaussian $W$ in $\mathbb R^K$ and $0<a<1$,
\[
 \Prob\{\|W\|\le a\sqrt K\}\le(e^{1/2}a)^K.
\]
Applied with $K=K_D$ and $a=a_D$, this gives the first term in
\eqref{eq:bad-Gaussian-probability}.  Similarly, for a standard Gaussian in
$M$ dimensions,
\[
 \Prob\{\|Z\|\ge R\}\le2^{M/2}e^{-R^2/4}.
\]
With $R=b_D$ and $M=O((\log D)^2)$ this gives the second term.
\end{proof}

\begin{proposition}\label{prop:transfer}
For $0<y\le s\le x_\star$, $0<\rho\le1$, and $L>0$, define
\begin{align}
 I_{\cP}(s,y,\rho;L)
 &=\E F_{s,y,\rho,L}(\widehat X),\label{eq:Iplateau-def}\\
 I_{\rm full}(s,y,\rho;L)
 &=\E F_{s,y,\rho,L}(\widehat Y_1).
 \label{eq:Ifull-def}
\end{align}
Then
\begin{align}
 I_{\cP}(s,y,\rho;L)
 &\le
 e^{M\sqrt D+\varepsilon_D}
 I_{\rm full}(s,y,\rho;L+2x_\star\delta_D)
 \label{eq:transfer}\\
 &\quad+
 s^{N}
 \left(e^{-cD^2\sqrt D}+e^{-ce^{\sqrt D}}\right),
 \label{eq:transfer-tail}
\end{align}
where
\begin{equation}\label{eq:epsilonD-explicit}
 \varepsilon_D=Cx_\star\rho^{-1}D^2\delta_D.
\end{equation}
In particular, if $\rho\ge D^{-B}$ for some fixed $B>0$, then
$\varepsilon_D=o(1)$ uniformly in $s$ and $y$.
\end{proposition}

\begin{proof}
On $\mathcal G_D$, Lemma~\ref{lem:root-lipschitz} gives the pointwise
one-sided comparison
\[
 F_{s,y,\rho,L}(\widehat X)
 \le e^{\varepsilon_D}
 F_{s,y,\rho,L+2x_\star\delta_D}(\widehat Y_\tau).
\]
On $\mathcal G_D^c$, use
$|\sinc|\le1$ and $\rho\le1$ to obtain the bound
$F_{s,y,\rho,L}\le s^N$.  Taking expectations,
using Lemma~\ref{lem:direction-perturbation} on the bad event, and then
applying Lemma~\ref{lem:density-domination} to the nonnegative enlarged
functional proves \eqref{eq:transfer}--\eqref{eq:transfer-tail}.
The final assertion follows because $\delta_D$ decays exponentially in
$\sqrt D$.
\end{proof}

We next estimate the change caused by omitting the low-weight component
in \eqref{eq:Hsplit}.

\begin{lemma}\label{lem:cheap-removal}
Let $0\le z<1$, $L>0$, $0<s\le x_\star$, and $0<\rho\le1$.  Let
$K_{\cP}$ and $K_{\cL}$ be Hilbert--Schmidt unit vectors in the high-weight and
low-weight Hermitian subspaces, and put
\begin{equation}\label{eq:Hz-definition}
 H_z=\sqrt{1-z}\,K_{\cP}+\frac{\sqrt z}{D}K_{\cL},
 \qquad
 y=s\sqrt{1-z}.
\end{equation}
Then
\begin{equation}\label{eq:cheap-determinant-comparison}
 \mathcal R_{s,\rho}(H_z)
 \le
 \exp\{Cx_\star\rho^{-1}D\sqrt M\}
 \mathcal R_{s,y,\rho}(K_{\cP}).
\end{equation}
Furthermore,
\begin{equation}\label{eq:cheap-phase-enlargement}
 s\diam\spec(H_z)\le L
 \quad\Longrightarrow\quad
 y\diam\spec(K_{\cP})
 \le L+\frac{2x_\star\sqrt M}{D}.
\end{equation}
\end{lemma}

\begin{proof}
The perturbation from $\sqrt{1-z}K_{\cP}$ to $H_z$ has operator norm at
most $\sqrt M/D$ by \eqref{eq:cheap-op-bound}.  If
$\kappa_1,\ldots,\kappa_D$ are the eigenvalues of $K_{\cP}$, Weyl's
inequality therefore changes each phase difference
$y(\kappa_i-\kappa_j)$ by at most
$2x_\star\sqrt M/D$.  The same scalar Lipschitz estimate used in
Lemma~\ref{lem:root-lipschitz}, summed over $O(D^2)$ real root factors,
gives \eqref{eq:cheap-determinant-comparison}.  The spectral-diameter
estimate follows from
\[
 \diam\spec(\sqrt{1-z}K_{\cP})
 \le\diam\spec(H_z)
   +2\left\|\frac{\sqrt z}{D}K_{\cL}\right\|_{\op}.
\]
\end{proof}

In the full-sphere integral obtained from Proposition~\ref{prop:phase-range},
Lemma~\ref{lem:cheap-removal}, and Proposition~\ref{prop:transfer}, the
retained phases have range at most
\begin{equation}\label{eq:total-phase-enlargement}
 2\pi+\frac{2x_\star\sqrt M}{D}+2x_\star\delta_D.
\end{equation}
Since $M=O((\log D)^2)$ and $\delta_D=e^{-\Omega(\sqrt D)}$, this is at most
the fixed number
\begin{equation}\label{eq:Lstar}
 L_\star=2\pi+1
\end{equation}
for all sufficiently large $D$.

\subsection{Abel estimates for the logarithmic kernel}

For $\rho>0$, define
\begin{equation}\label{eq:Krho}
 K_\rho(t)=\max\{2|\sin(t/2)|,\rho|t|\}.
\end{equation}

\begin{lemma}\label{lem:abel}
Fix $L>2\pi$.  There are universal $c_0,c_1>0$ with the following property.
Suppose $0<\Delta\le1$ and real numbers $\theta_1,\ldots,\theta_D$ satisfy
\begin{equation}\label{eq:abel-hyp}
 \max_{i,j}|\theta_i-\theta_j|\le L,
 \qquad
 \left|
 \frac1D\sum_i\operatorname{Arg}(e^{i\theta_i})^2
 -\frac{\pi^2}{3}
 \right|\ge\Delta.
\end{equation}
For all sufficiently small $\Delta$, choose
\[
 h=c_0\Delta^2,
 \qquad r=e^{-h},
 \qquad \rho=(1-r)/L.
\]
Then
\begin{equation}\label{eq:abel-product}
 \log\prod_{i<j}K_\rho(\theta_i-\theta_j)^2
 \le-c_1\Delta^2D^2+C D\log(1/\Delta).
\end{equation}
For fixed $0<\Delta\le1$, the same conclusion holds with a fixed $\rho$
and an $O_\Delta(D)$ remainder.
\end{lemma}

\begin{proof}
For $|t|\le L$,
\[
 K_\rho(t)\le r^{-1/2}|1-re^{it}|.
\]
Let $p_k=\sum_je^{ik\theta_j}$.  The exact Abel identity is
\begin{equation}\label{eq:abel-identity}
 \log\prod_{i<j}|1-re^{i(\theta_i-\theta_j)}|^2
 =-\sum_{k\ge1}\frac{r^k}{k}|p_k|^2-D\log(1-r).
\end{equation}
Writing $\phi_j=\operatorname{Arg}(e^{i\theta_j})$ and using
\[
 \phi^2=\frac{\pi^2}{3}
 +4\sum_{k\ge1}\frac{(-1)^k}{k^2}\cos(k\phi)
\]
gives
\[
 \left|
 \sum_{k\ge1}\frac{(-1)^{k+1}}{k^2}\Re p_k
 \right|
 \ge\frac{D\Delta}{4}.
\]
Set
$A(r)=\sum_{k\ge1}(1-r^{k/2})/k^2$.  Since
$1-e^{-hk/2}\le\min\{hk/2,1\}$,
\[
 A(e^{-h})\le Ch\log(e/h).
\]
For $h=c_0\Delta^2$ with $c_0$ sufficiently small, this is at most
$\Delta/8$.  Hence
\[
 \left|\sum_{k\ge1}
 \frac{(-1)^{k+1}r^{k/2}}{k^2}\Re p_k\right|
 \ge\frac{D\Delta}{8}.
\]
Cauchy--Schwarz yields
\begin{equation}\label{eq:fourier-energy}
 \sum_{k\ge1}\frac{r^k}{k}|p_k|^2
 \ge\frac{D^2\Delta^2}{64\zeta(3)}.
\end{equation}
Substituting the majorant into \eqref{eq:abel-identity} introduces the error
terms
\[
 -\binom D2\log r=O(hD^2),
 \qquad
 -D\log(1-r)=O(D\log(1/\Delta)).
\]
Taking $c_0$ sufficiently small relative to the coefficient in
\eqref{eq:fourier-energy} proves \eqref{eq:abel-product}.
\end{proof}

\begin{lemma}
\label{lem:abel-sharp-coefficient}
Fix $L>2\pi$ and $\Delta>0$.  Suppose, for each $D$, that real numbers
$\theta_1,\ldots,\theta_D$ satisfy
\begin{equation}\label{eq:abel-sharp-hyp}
 \max_{i,j}|\theta_i-\theta_j|\le L,
 \qquad
 \left|
 \frac1D\sum_i\operatorname{Arg}(e^{i\theta_i})^2-\frac{\pi^2}{3}
 \right|\ge\Delta.
\end{equation}
Let $h_D\downarrow0$ satisfy $\log(1/h_D)=o(D)$, put
$r_D=e^{-h_D}$, and set $\rho_D=(1-r_D)/L$.  Then
\begin{equation}\label{eq:abel-sharp-conclusion}
 \limsup_{D\to\infty}\frac1{D^2}
 \log\prod_{i<j}K_{\rho_D}(\theta_i-\theta_j)^2
 \le-\frac{\Delta^2}{16\zeta(3)}.
\end{equation}
\end{lemma}

\begin{proof}
With $p_k=\sum_j e^{ik\theta_j}$, the Fourier identity in the proof of
Lemma~\ref{lem:abel} gives
\[
 \left|\sum_{k\ge1}\frac{(-1)^{k+1}}{k^2}\Re p_k\right|
 \ge\frac{D\Delta}{4}.
\]
Since
$\sum_{k\ge1}(1-r_D^{k/2})/k^2=o(1)$, the same linear functional with the
factor $r_D^{k/2}$ has absolute value at least
$D(\Delta/4-o(1))$.  Cauchy--Schwarz therefore yields
\[
 \sum_{k\ge1}\frac{r_D^k}{k}|p_k|^2
 \ge D^2\left(\frac{\Delta^2}{16\zeta(3)}-o(1)\right).
\]
Substituting the majorant
$K_{\rho_D}(t)\le r_D^{-1/2}|1-r_De^{it}|$ into the Abel identity introduces
only
\[
 -\binom D2\log r_D=o(D^2),
 \qquad
 -D\log(1-r_D)=o(D^2),
\]
by the hypotheses on $h_D$.  This proves
\eqref{eq:abel-sharp-conclusion}.
\end{proof}

\subsection{Weyl integration and normalization}

Let
\[
 \Sigma_D=\left\{\lambda\in\mathbb R^D:
 \sum_i\lambda_i=0,\ \sum_i\lambda_i^2=D\right\}
\]
with normalized surface probability $d\nu_D$, and write
$P_D=\binom D2$.  Set
\[
 \Delta(\lambda)=\prod_{i<j}(\lambda_i-\lambda_j)
\]
and define
\begin{equation}\label{eq:ZD}
 \mathcal Z_D=\int_{\Sigma_D}\Delta(\lambda)^2\,d\nu_D(\lambda).
\end{equation}
Write $\E_{\rm full}$ for normalized expectation on the full trace-zero
Hermitian unit sphere.  By the Weyl eigenvalue decomposition on the trace-zero
Hermitian space, followed by radial disintegration \cite{Hall,Mehta}, every
nonnegative conjugation-invariant function $\Psi$ on that sphere satisfies
\begin{equation}\label{eq:spherical-Weyl-formula}
 \E_{\rm full}\Psi(H)
 =\frac1{\mathcal Z_D}
  \int_{\Sigma_D}\Psi(\diag\lambda)\Delta(\lambda)^2\,d\nu_D(\lambda).
\end{equation}
This normalization makes \eqref{eq:spherical-Weyl-formula} a probability
identity.

\begin{proposition}\label{prop:ZD}
One has
\begin{equation}\label{eq:ZDexact}
 \mathcal Z_D=
 2^{-P_D}D^{P_D}
 \frac{\Gamma((D-1)/2)}{\Gamma((D^2-1)/2)}
 \prod_{j=1}^D j!.
\end{equation}
\end{proposition}

\begin{proof}
The beta-two Mehta integral \cite{Mehta} is
\[
 \int_{\mathbb R^D}e^{-\|\lambda\|^2/2}\Delta(\lambda)^2\,d\lambda
 =(2\pi)^{D/2}\prod_{j=1}^D j!.
\]
Splitting off the normalized trace coordinate gives the trace-zero Gaussian
integral
\[
 (2\pi)^{(D-1)/2}\prod_{j=1}^D j!.
\]
In the $(D-1)$-dimensional trace-zero hyperplane, $\Delta^2$ has degree
$D(D-1)$, so the radial integral is
\[
 2^{(D^2-3)/2}\Gamma((D^2-1)/2).
\]
Divide by the unit-sphere area
$2\pi^{(D-1)/2}/\Gamma((D-1)/2)$ and scale the radius from one to $\sqrt D$.
This gives \eqref{eq:ZDexact}.
\end{proof}

For $y=s\sqrt{1-z}$, a regularized root factor can be written
\begin{equation}\label{eq:root-K}
 s\max\left\{
 \left|\sinc\frac{y(\lambda_i-\lambda_j)}2\right|,\rho
 \right\}
 =\frac{K_\rho(y(\lambda_i-\lambda_j))}
 {\sqrt{1-z}|\lambda_i-\lambda_j|}.
\end{equation}
For $z<1$, the quotient is understood by continuity when eigenvalues
coincide.  Weyl's $\Delta^2$ density therefore cancels the Jacobi denominators:
\begin{equation}\label{eq:weyl-cancel}
 \Delta(\lambda)^2\prod_{i<j}
 \left[
 \frac{K_\rho(y(\lambda_i-\lambda_j))}
 {\sqrt{1-z}|\lambda_i-\lambda_j|}
 \right]^2
 =(1-z)^{-P_D}\prod_{i<j}K_\rho(y(\lambda_i-\lambda_j))^2.
\end{equation}

Let $\Vol_{\rm F}$ denote Riemannian volume for the Frobenius inner
product $\langle X,Y\rangle_{\rm F}=\Tr(X^\dagger Y)$ on $U(D)$.
Let $G_{\mathrm B}$ denote the Barnes $G$-function, so that
$G_{\mathrm B}(D+1)=\prod_{j=1}^{D-1}j!$.  Macdonald's compact-group volume formula
\cite{Macdonald} gives
\[
 \Vol_{\rm F}U(D)=\frac{(2\pi)^{D(D+1)/2}}{G_{\mathrm B}(D+1)}.
\]
The central $U(1)$ fibre has Frobenius length $2\pi\sqrt D$.  The quotient
has dimension $N=D^2-1$, and rescaling its metric by $D^{-1}$ multiplies its
volume by $D^{-N/2}$.  Therefore the normalized Hilbert--Schmidt group volume
is
\begin{equation}\label{eq:V0}
 \Vol_0\PU(D)
 =D^{-D^2/2}\frac{(2\pi)^{(D^2+D-2)/2}}{G_{\mathrm B}(D+1)}.
\end{equation}
Combining \eqref{eq:ZDexact}, \eqref{eq:V0}, and
$N\omega_N=2\pi^{N/2}/\Gamma(N/2)$ gives the exact identity
\begin{equation}\label{eq:exact-collapse}
 \frac{N\omega_N}{\Vol_0\PU(D)\,\mathcal Z_D}
 =\frac{2^{2-D}D^{D/2}\pi^{(1-D)/2}}
 {D!\Gamma((D-1)/2)}.
\end{equation}
Its logarithm is $-D\log D+O(D)$.

\section{Small-ball estimates and the lower threshold}\label{sec:small-ball}

We collect the preceding estimates into a single inequality with constants
uniform up to $x_\star=\pi/\sqrt3$.

Set
\begin{equation}\label{eq:C-D-definition}
 \mathfrak C_D
 :=\frac{2^{2-D}D^{D/2}\pi^{(1-D)/2}}
 {D!\Gamma((D-1)/2)},
\end{equation}
and
\begin{equation}\label{eq:B-D-definition}
 \mathfrak B_D
 :=\frac{B(M/2,(D-1-M)/2)}
 {B(M/2,(N-M)/2)}.
\end{equation}
Here $\mathfrak C_D$ is the normalization factor in
\eqref{eq:exact-collapse}, while $\mathfrak B_D$ is the beta-integral factor.

\begin{proposition}\label{prop:uniform-estimate}
There are absolute constants $\Delta_0,c,C>0$ such that the following holds
for all sufficiently large $D$.  Suppose
\begin{equation}\label{eq:uniform-hypotheses}
 D^{-1/4}\le\Delta\le\Delta_0,
 \qquad
 0<x\le x_\star,
 \qquad
 x^2\le\frac{\pi^2}{3}-\Delta.
\end{equation}
Use Lemma~\ref{lem:abel} with the fixed phase window $L_\star=2\pi+1$:
\begin{equation}\label{eq:uniform-regularizer}
 h=c_0\Delta^2,
 \qquad r=e^{-h},
 \qquad \rho=\frac{1-r}{L_\star}.
\end{equation}
Then $\rho\asymp\Delta^2$, with absolute comparison constants, and
\begin{align}
 \mu_D\!\left(B_A([I],xD)\right)
 &\le
 \mathfrak C_D\mathfrak B_D
 \frac{x^{D-1}}{D-1}
 \exp\left\{-c\Delta^2D^2+\mathcal E_D(\Delta)\right\}
 +\mathfrak T_D,
 \label{eq:uniform-integrated-bound}
\end{align}
where
\begin{equation}\label{eq:uniform-error-estimate}
 \mathcal E_D(\Delta)
 \le C\left(
 \Delta^{-2}MD^{3/2}
 +MD^{3/2}
 +\Delta^{-2}D\sqrt M
 +M\sqrt D
 +D\log(1/\Delta)
 \right)
\end{equation}
and
\begin{equation}\label{eq:uniform-tail}
 0\le\mathfrak T_D\le e^{-cD^2\sqrt D}.
\end{equation}
All constants in \eqref{eq:uniform-integrated-bound}--
\eqref{eq:uniform-tail} are uniform in $x$, $\Delta$, and $D$ in the stated
range.
\end{proposition}

\begin{proof}
We apply the successive comparisons to the exact polar identity
\eqref{eq:polar-normalized}.  Proposition~\ref{prop:actual-frozen} first
bounds the Jacobi determinant by the regularized determinant of the
constant-coefficient model.
Thus, for every $0<s\le x$,
\begin{equation}\label{eq:uniform-step-Jacobi}
 \mathbf1_{\{sD<c(X)\}}|\det\widehat J(sD,X)|
 \le e^{E_J}
 \mathbf1_{\{sD<c(X)\}}
 \det_{\rho,s}\widehat J_0(sD,X),
\end{equation}
where
\begin{equation}\label{eq:EJ-definition}
 E_J\le C\left(\rho^{-1}MD^{3/2}+MD^{3/2}\right)
\end{equation}
uniformly for $s\le x_\star$.  By
\eqref{eq:regularized-frozen-product}, after conditioning on the beta
variable $z$ in \eqref{eq:Hsplit}, this regularized determinant is the
product $\mathcal R_{s,\rho}(H)$ defined in
\eqref{eq:root-functional-general}.  We next remove the low-weight component,
transfer the high-weight average to the full Hermitian sphere, and then apply Weyl
integration and the Abel estimate.

Proposition~\ref{prop:phase-range} gives
$s\diam\spec(H)<2\pi$ on the minimizing domain.  Apply
Lemma~\ref{lem:cheap-removal} with
$y=s\sqrt{1-z}$.  It removes the low-weight term at logarithmic cost
\begin{equation}\label{eq:Echeap-definition}
 E_{\rm low}\le C\rho^{-1}D\sqrt M
\end{equation}
and increases the phase range by at most
$2x_\star\sqrt M/D$.  Proposition~\ref{prop:transfer} then compares this
average with the
full Hermitian-sphere average at logarithmic cost
\begin{equation}\label{eq:Etransfer-definition}
 E_{\rm transfer}\le M\sqrt D+\varepsilon_D,
 \qquad
 \varepsilon_D=Cx_\star\rho^{-1}D^2\delta_D=o(1).
\end{equation}
Because of \eqref{eq:total-phase-enlargement}, the resulting full-sphere
integrand has phase diameter at most the fixed $L_\star$ for large $D$.
The assumption $\Delta\ge D^{-1/4}$ implies that $\rho^{-1}$ is polynomially
bounded, so the uniform $o(1)$ assertion for $\varepsilon_D$ applies.

Let $K$ now denote a full Hermitian Hilbert--Schmidt unit vector and let
$\lambda\in\Sigma_D$ be its eigenvalue vector.  Under the spherical Weyl
formula \eqref{eq:spherical-Weyl-formula}, the squared Vandermonde density
cancels the root denominators by \eqref{eq:weyl-cancel}.  Since
$y=s\sqrt{1-z}$,
\begin{equation}\label{eq:second-moment-bound}
 \frac1D\sum_i\Arg(e^{iy\lambda_i})^2
 \le\frac1D\sum_i(y\lambda_i)^2
 =y^2\le s^2\le x^2
 \le\frac{\pi^2}{3}-\Delta.
\end{equation}
Thus Lemma~\ref{lem:abel}, with $L=L_\star$, gives the pointwise estimate
\begin{equation}\label{eq:uniform-Abel-bound}
 \prod_{i<j}K_\rho(y(\lambda_i-\lambda_j))^2
 \le
 \exp\{-c\Delta^2D^2+CD\log(1/\Delta)\}.
\end{equation}
The Cartan singular values contribute $s^{D-1}$.  Therefore the full-sphere
average, after Weyl cancellation, is bounded by
\begin{equation}\label{eq:full-average-bound}
 \frac{s^{D-1}}{\mathcal Z_D}
 (1-z)^{-P_D}
 \exp\{-c\Delta^2D^2+CD\log(1/\Delta)\}.
\end{equation}

The remaining integrations are scalar.  The beta density in \eqref{eq:beta},
multiplied by $(1-z)^{-P_D}$, has integral
\begin{align}
 &\frac1{B(M/2,(N-M)/2)}
 \int_0^1
 z^{M/2-1}(1-z)^{(N-M)/2-1-P_D}\,dz
 \notag\\
 &\hspace{35mm}
 =\frac{B(M/2,(D-1-M)/2)}
 {B(M/2,(N-M)/2)}
 =\mathfrak B_D.
 \label{eq:beta-ratio}
\end{align}
Since $(N-M)/2-P_D=(D-1-M)/2$, the beta integral is finite for all
sufficiently large $D$.  The radial integral is
\begin{equation}\label{eq:s-int}
 \int_0^x s^{D-1}\frac{ds}{s}
 =\frac{x^{D-1}}{D-1}.
\end{equation}
Finally, the polar prefactor together with $\mathcal Z_D^{-1}$ is precisely
$\mathfrak C_D$ by \eqref{eq:exact-collapse}.  Combining
\eqref{eq:EJ-definition}, \eqref{eq:Echeap-definition},
\eqref{eq:Etransfer-definition}, and \eqref{eq:uniform-Abel-bound}, and using
$\rho^{-1}=O(\Delta^{-2})$, gives
\eqref{eq:uniform-error-estimate}.

For the bad Gaussian event, Proposition~\ref{prop:transfer} retains the
factor $s^N$.  For suitable absolute constants $c_1,c_2>0$, write
\[
 p_D=e^{-c_1D^2\sqrt D}+e^{-c_2e^{\sqrt D}}.
\]
Its contribution is bounded by
\begin{align}
 \mathfrak T_D
 &\le
 \frac{N\omega_N}{\Vol_0\PU(D)}
 \exp\{E_J+E_{\rm low}\}
 p_D\int_0^{x_\star}s^{N-1}\,ds \notag\\
 &=
 \frac{N\omega_N}{\Vol_0\PU(D)}
 \exp\{E_J+E_{\rm low}\}
 p_D\frac{x_\star^N}{N}.
 \label{eq:bad-tail-full-bound}
\end{align}
Under \eqref{eq:uniform-hypotheses},
$\rho^{-1}=O(D^{1/2})$, and hence
\[
 E_J=O(MD^2),\qquad
 E_{\rm low}=O(D^{3/2}\sqrt M).
\]
The exact group-volume formula and the Euclidean sphere-area formula give
\[
 \log\frac{N\omega_N}{\Vol_0\PU(D)}=O(D^2\log D),
 \qquad
 N\log x_\star=O(D^2).
\]
Because $M=O((\log D)^2)$, every positive term in the logarithm of
\eqref{eq:bad-tail-full-bound} is $o(D^{5/2})$, whereas the first summand in
$p_D$ is $e^{-c_1D^{5/2}}$ and the second is smaller.  After decreasing the
absolute constant, this proves \eqref{eq:uniform-tail}.
\end{proof}

The exact scalar factors satisfy
\begin{equation}\label{eq:scalar-factor-asymptotics}
 \log\mathfrak C_D=-D\log D+O(D),
 \qquad
 \log\mathfrak B_D
 =\frac M2\log D+O\left(M+\frac{M^2}{D}\right).
\end{equation}
The first estimate follows from Stirling's formula.  For the second, write
$a=M/2$, $b=(N-M)/2$, and $c=(D-1-M)/2$ and use, uniformly for $a=o(t)$,
\[
 \log\Gamma(t+a)-\log\Gamma(t)
 =a\log t+O(a^2/t+a/t).
\]
Uniformly for $x$ in every compact subinterval of $(0,x_\star]$,
\begin{equation}\label{eq:radial-factor-asymptotic}
 \log\frac{x^{D-1}}{D-1}=O(D).
\end{equation}

\begin{proposition}\label{prop:fixed}
For every fixed $0<x<\pi/\sqrt3$, there are constants $c_x>0$ and $D_x$ such
that, for every power of two $D\ge D_x$,
\begin{equation}\label{eq:fixed-main}
 \mu_D\!\left(B_A([I],xD)\right)\le e^{-c_xD^2}.
\end{equation}
\end{proposition}

\begin{proof}
Fix $0<x<x_\star$ and set
\begin{equation}\label{eq:Delta-x}
 \Delta_x
 =\min\left\{\Delta_0,\frac{\pi^2}{3}-x^2\right\}>0.
\end{equation}
For large $D$, the hypotheses of Proposition~\ref{prop:uniform-estimate} hold
with $\Delta=\Delta_x$.  Since $M=O((\log D)^2)$,
\[
 \mathcal E_D(\Delta_x)=O_x(MD^{3/2}+D\log D)=o(D^2).
\]
The logarithms of the exact beta, radial, and normalization factors are also
$o(D^2)$ by \eqref{eq:scalar-factor-asymptotics} and
\eqref{eq:radial-factor-asymptotic}.  Hence there is $D_x$ such that, for all
powers of two $D\ge D_x$, the first term in
\eqref{eq:uniform-integrated-bound} is at most
\[
 \exp\left\{-\frac{c\Delta_x^2}{2}D^2\right\}.
\]
The tail $\mathfrak T_D\le e^{-cD^{5/2}}$ is smaller after increasing $D_x$.
Absorbing the sum of the two exponentials into the constant and setting, for
example, $c_x=c\Delta_x^2/4$, gives
\[
 \mu_D\!\left(B_A([I],xD)\right)\le e^{-c_xD^2},
\]
as claimed.
\end{proof}

\begin{proposition}\label{prop:calibrated}
There are positive absolute constants
\[
 A_*,\qquad c_*,\qquad C_*,\qquad D_*
\]
with the following property.  Set
\begin{equation}\label{eq:delta-calibrated}
 \widetilde\Delta_D=\left(\frac{A_*M}{D^{1/2}}\right)^{1/4},
 \qquad
 \widetilde x_D=\sqrt{\frac{\pi^2}{3}-\widetilde\Delta_D}.
\end{equation}
Then, for every power of two $D\ge D_*$,
\begin{equation}\label{eq:calibrated-volume}
 \mu_D\!\left(B_A([I],\widetilde x_DD)\right)
 \le \exp\{-c_*M^{1/2}D^{7/4}\}.
\end{equation}
Consequently, a Haar-random $[U]$ satisfies
\begin{equation}\label{eq:calibrated-distance}
 d_A([I],[U])\ge
 \left(\frac\pi{\sqrt3}
 -C_*D^{-1/8}(\log D)^{1/2}\right)D
\end{equation}
with probability at least
\begin{equation}\label{eq:calibrated-probability}
 1-\exp\{-c_*M^{1/2}D^{7/4}\}.
\end{equation}
\end{proposition}

\begin{proof}
Let
\[
 \Delta=\left(\frac{aM}{D^{1/2}}\right)^{1/4}
\]
with $a>0$ fixed.  The two leading terms in the uniform estimate have the form
\begin{equation}\label{eq:calibrated-leading-terms}
 -c_{\rm A}\Delta^2D^2
 =-c_{\rm A}a^{1/2}M^{1/2}D^{7/4},
\end{equation}
\begin{equation}\label{eq:calibrated-comparison-term}
 C_{\rm J}\Delta^{-2}MD^{3/2}
 =C_{\rm J}a^{-1/2}M^{1/2}D^{7/4},
\end{equation}
where $c_{\rm A},C_{\rm J}>0$ are absolute constants furnished by
Proposition~\ref{prop:uniform-estimate}.  Choose $a=A_*$ so large that
\[
 c_{\rm A}A_*^{1/2}-C_{\rm J}A_*^{-1/2}>0.
\]
All remaining terms in \eqref{eq:uniform-error-estimate} and all scalar
normalization terms are $o(M^{1/2}D^{7/4})$.  Thus, after fixing $A_*$, there
are absolute constants $c_*>0$ and $D_*$ such that
\[
 \mu_D\!\left(B_A([I],\widetilde x_DD)\right)
 \le e^{-c_*M^{1/2}D^{7/4}}
\]
for every power of two $D\ge D_*$.  Finally,
\[
 0<\frac\pi{\sqrt3}-\widetilde x_D
 \le\frac{\sqrt3}{\pi}\widetilde\Delta_D
 =\frac{\sqrt3}{\pi}A_*^{1/4}D^{-1/8}M^{1/4}.
\]
Since $M^{1/4}\le C(\log D)^{1/2}$, enlarging an absolute constant $C_*$
gives \eqref{eq:calibrated-distance}.  The same volume estimate gives the
failure-probability bound in \eqref{eq:calibrated-probability}.
\end{proof}

\begin{corollary}\label{cor:diagonal}
Let
\begin{equation}\label{eq:delta-diag}
 \Delta_D=\left(\frac{M\log\log D}{D^{1/2}}\right)^{1/4},
 \qquad
 x_D=\sqrt{\frac{\pi^2}{3}-\Delta_D}.
\end{equation}
There are absolute constants $c_{\rm d},C_{\rm d}>0$ and $D_{\rm d}$ such that, for every power
of two $D\ge D_{\rm d}$,
\begin{equation}\label{eq:diag-volume}
 \mu_D\!\left(B_A([I],x_DD)\right)
 \le \exp\{-c_{\rm d}(M\log\log D)^{1/2}D^{7/4}\}.
\end{equation}
Consequently, a Haar-random $[U]$ satisfies
\begin{equation}\label{eq:distance-main}
 d_A([I],[U])\ge
 \left(\frac\pi{\sqrt3}
 -C_{\rm d}D^{-1/8}(\log D)^{1/2}(\log\log D)^{1/4}\right)D
\end{equation}
with probability at least
\begin{equation}\label{eq:diagonal-probability}
 1-\exp\{-c_{\rm d}(M\log\log D)^{1/2}D^{7/4}\}.
\end{equation}
\end{corollary}

\begin{proof}
For this choice of $\Delta_D$, the hypotheses of
Proposition~\ref{prop:uniform-estimate} hold for all sufficiently large $D$.
The negative term and the largest comparison error have respective sizes
\[
 \Delta_D^2D^2=(M\log\log D)^{1/2}D^{7/4},
 \qquad
 \Delta_D^{-2}MD^{3/2}=(M/\log\log D)^{1/2}D^{7/4}.
\]
Their ratio is $\log\log D$, and all other terms in
\eqref{eq:uniform-integrated-bound} are lower order.  This proves
\eqref{eq:diag-volume}.  Finally,
\[
 0<\frac\pi{\sqrt3}-x_D
 =\frac{\Delta_D}{\pi/\sqrt3+\sqrt{\pi^2/3-\Delta_D}}
 \le\frac{\sqrt3}{\pi}\Delta_D,
\]
and $M=O((\log D)^2)$ gives \eqref{eq:distance-main} with the probability
in \eqref{eq:diagonal-probability}.
\end{proof}

\section{CUE estimates and the distance threshold}\label{sec:cue-upper}

\subsection{Second moments and logarithmic paths}

Let $U$ be Haar distributed on $U(D)$ and choose principal eigenangles
$\theta_1,\ldots,\theta_D\in(-\pi,\pi]$.  Its projective class $[U]$ is Haar
distributed on $\PU(D)$.  Define the empirical principal-angle second moment
\begin{equation}\label{eq:CUE-m2}
 m_2(U)=\frac1D\sum_{j=1}^D\theta_j^2.
\end{equation}

\begin{proposition}\label{prop:CUE-m2}
There are absolute constants $\Delta_{\rm C},c_{\rm C}>0$ such that, for all
sufficiently large $D$ and every
\[
 D^{-1/4}\le\Delta\le\Delta_{\rm C},
\]
one has
\begin{equation}\label{eq:CUE-m2-tail}
 \Prob\!\left\{
 \left|m_2(U)-\frac{\pi^2}{3}\right|\ge\Delta
 \right\}
 \le \exp\{-c_{\rm C}\Delta^2D^2\}.
\end{equation}
For every fixed $\Delta>0$, the same conclusion holds in the form
$e^{-c_\Delta D^2}$ for all sufficiently large $D$.
\end{proposition}

\begin{proof}
Weyl integration on $U(D)$ gives the joint density of the principal
angles on $(-\pi,\pi]^D$:
\begin{equation}\label{eq:CUE-density}
 \frac{1}{(2\pi)^DD!}
 \prod_{i<j}|e^{i\theta_i}-e^{i\theta_j}|^2
 =\frac{1}{(2\pi)^DD!}
 \prod_{i<j}\left(2\left|\sin\frac{\theta_i-\theta_j}{2}\right|\right)^2.
\end{equation}
On the event in \eqref{eq:CUE-m2-tail}, the lifted phase range is at most
$2\pi<L_\star$, and the second hypothesis of Lemma~\ref{lem:abel} holds.
Since $2|\sin(t/2)|\le K_\rho(t)$, that lemma gives the pointwise bound
\[
 \prod_{i<j}|e^{i\theta_i}-e^{i\theta_j}|^2
 \le
 \exp\{-c_1\Delta^2D^2+CD\log(1/\Delta)\}.
\]
Integrating over the event and using that the ambient cube has volume
$(2\pi)^D$ yields
\begin{equation}\label{eq:CUE-tail-preabsorb}
 \Prob\!\left\{
 \left|m_2(U)-\frac{\pi^2}{3}\right|\ge\Delta
 \right\}
 \le
 \frac1{D!}
 \exp\{-c_1\Delta^2D^2+CD\log(1/\Delta)\}.
\end{equation}
If $\Delta\ge D^{-1/4}$, the positive remainder is
$o(\Delta^2D^2)$ uniformly in the stated range, which proves
\eqref{eq:CUE-m2-tail}.  For a fixed larger deviation, apply the estimate at
$\min\{\Delta,\Delta_{\rm C}\}$.
\end{proof}

\begin{proposition}\label{prop:log-path}
For every $U\in U(D)$ with principal eigenangles as above, the
centered principal logarithm defines a one-parameter subgroup path $\gamma_U$
from $[I]$ to $[U]$ such that
\begin{equation}\label{eq:log-path-upper}
 d_A([I],[U])\le\ell_A(\gamma_U)
 \le D\sqrt{m_2(U)}\le\pi D.
\end{equation}
\end{proposition}

\begin{proof}
Write
\[
 U=V\diag(e^{i\theta_1},\ldots,e^{i\theta_D})V^*,
 \qquad
 \overline\theta=\frac1D\sum_j\theta_j,
\]
and set
\[
 K_U=V\diag(\theta_1-\overline\theta,\ldots,
             \theta_D-\overline\theta)V^*.
\]
Then $K_U$ is traceless and $e^{iK_U}=e^{-i\overline\theta}U$, so
$[e^{iK_U}]=[U]$.  Define
$\gamma_U:[0,1]\to\PU(D)$ by $\gamma_U(t)=[e^{itK_U}]$.  Its right-trivialized
velocity is the constant $iK_U$.  Since $A\le D^2I$,
\[
 \ell_A(\gamma_U)=\|iK_U\|_A
 \le D\|K_U\|_0.
\]
Moreover,
\[
 \|K_U\|_0^2
 =\frac1D\sum_j(\theta_j-\overline\theta)^2
 \le\frac1D\sum_j\theta_j^2=m_2(U).
\]
The final bound follows from $|\theta_j|\le\pi$.
\end{proof}

\begin{proof}[Proof of Theorem~\ref{thm:sharp}]
Let $\eta_D$ be as in \eqref{eq:eta-sharp}.  Since
$M\asymp(\log D)^2$, Proposition~\ref{prop:calibrated} implies, after changing
absolute constants,
\begin{equation}\label{eq:sharp-lower-tail}
 \Prob\!\left\{
 \frac{d_A([I],[U])}{D}
 <\frac{\pi}{\sqrt3}-C_-\eta_D
 \right\}
 \le e^{-c_-M^{1/2}D^{7/4}}.
\end{equation}
For the upper tail, $\eta_D\to0$ and $\eta_D\ge D^{-1/4}$ for all large $D$.
On the event
\[
 m_2(U)\le\frac{\pi^2}{3}+\eta_D,
\]
Proposition~\ref{prop:log-path} and concavity of the square root give
\[
 \frac{d_A([I],[U])}{D}
 \le\sqrt{\frac{\pi^2}{3}+\eta_D}
 \le\frac{\pi}{\sqrt3}+C_+\eta_D.
\]
Proposition~\ref{prop:CUE-m2} therefore yields
\begin{equation}\label{eq:sharp-upper-tail}
 \Prob\!\left\{
 \frac{d_A([I],[U])}{D}
 >\frac{\pi}{\sqrt3}+C_+\eta_D
 \right\}
 \le e^{-c_+\eta_D^2D^2}
 =e^{-c_+M^{1/2}D^{7/4}}.
\end{equation}
Combining \eqref{eq:sharp-lower-tail} and
\eqref{eq:sharp-upper-tail} proves \eqref{eq:sharp-concentration}.
The global bound \eqref{eq:log-path-upper} gives
$0\le d_A([I],[U])/D\le\pi$.  Hence convergence in probability together with
uniform boundedness implies convergence in every finite $L^p$.
\end{proof}

\subsection{Large deviations and ball probabilities}

Let $\mathcal P(\mathbb T)$ denote the probability measures on the unit
circle with the weak topology, and let $m_{\mathbb T}$ be normalized arc
length.  The following is
the Haar-unitary specialization of the large-deviation principle of Hiai and
Petz \cite{HiaiPetz}.

\begin{theorem}[Hiai--Petz]\label{thm:CUE-LDP}
Let $e^{i\theta_1},\ldots,e^{i\theta_D}$ be the eigenvalues of a Haar
unitary in $U(D)$, and set
\[
 L_D=\frac1D\sum_{j=1}^D\delta_{e^{i\theta_j}}.
\]
Here $\delta_z$ denotes the unit point mass at $z$.  The sequence
$(L_D)$ satisfies a large-deviation principle on $\mathcal P(\mathbb T)$ with
speed $D^2$ and good rate function
\begin{equation}\label{eq:CUE-rate-energy}
 \mathcal I(\nu)=
 -\iint_{\mathbb T^2}\log|z-w|\,d\nu(z)d\nu(w),
\end{equation}
with the value $+\infty$ when the logarithmic energy is not finite.  The
normalization is such that $\mathcal I(m_{\mathbb T})=0$.  In particular, for
every open set $G\subset\mathcal P(\mathbb T)$ and every closed set
$F\subset\mathcal P(\mathbb T)$,
\[
 \liminf_{D\to\infty}\frac1{D^2}\log\Prob\{L_D\in G\}
 \ge-\inf_{\nu\in G}\mathcal I(\nu),
\]
\[
 \limsup_{D\to\infty}\frac1{D^2}\log\Prob\{L_D\in F\}
 \le-\inf_{\nu\in F}\mathcal I(\nu).
\]
\end{theorem}

\begin{lemma}
\label{lem:CUE-energy-Fourier}
For $\nu\in\mathcal P(\mathbb T)$ and $0<r<1$, set
\[
 \mathcal I_r(\nu)
 =-\iint_{\mathbb T^2}\log|1-rz\overline w|\,d\nu(z)d\nu(w).
\]
Then
\begin{equation}\label{eq:CUE-Abel-energy}
 \mathcal I_r(\nu)
 =\sum_{k\ge1}\frac{r^k}{k}|\widehat\nu_k|^2,
 \qquad
 \widehat\nu_k=\int_{\mathbb T}z^k\,d\nu(z).
\end{equation}
If $\mathcal I(\nu)<\infty$, then
\begin{equation}\label{eq:CUE-rate-Fourier}
 \mathcal I(\nu)
 =\lim_{r\uparrow1}\mathcal I_r(\nu)
 =\sum_{k\ge1}\frac{|\widehat\nu_k|^2}{k}.
\end{equation}
\end{lemma}

\begin{proof}
For $0<r<1$ the uniformly convergent Fourier expansion
\[
 -\log|1-re^{it}|=\sum_{k\ge1}\frac{r^k}{k}\cos(kt)
\]
may be integrated against $\nu\otimes\nu$, which gives
\eqref{eq:CUE-Abel-energy}.  The right-hand side increases to the series in
\eqref{eq:CUE-rate-Fourier}.  Abel convergence for the circular logarithmic
kernel identifies this limit with \eqref{eq:CUE-rate-energy} whenever the
logarithmic energy is finite; if the series diverges, the energy is infinite.
\end{proof}

\begin{proof}[Proof of Theorem~\ref{thm:ball-probabilities}]
For $0<x<\pi/\sqrt3$, the small-ball upper bound is
Proposition~\ref{prop:fixed}.  To prove the probability lower bound, define
\[
 \vartheta(e^{i\theta})=\theta^2,
 \qquad \theta\in(-\pi,\pi],
\]
with $\vartheta(-1)=\pi^2$.  The one-sided limits at $-1$ agree, so
$\vartheta$ is continuous on $\mathbb T$.  Choose a probability measure
$\nu_x$ with finite logarithmic energy and
$\int\vartheta\,d\nu_x<x^2$; normalized arc length on a sufficiently short
symmetric arc is one example.  The set
\[
 G_x=\left\{\nu\in\mathcal P(\mathbb T):
             \int\vartheta\,d\nu<x^2\right\}
\]
is weakly open.  Theorem~\ref{thm:CUE-LDP} gives a finite constant $C_x$ such
that, for all sufficiently large $D$,
\begin{equation}\label{eq:CUE-lower-m2}
 \Prob\{m_2(U)<x^2\}=\Prob\{L_D\in G_x\}\ge e^{-C_xD^2}.
\end{equation}
Proposition~\ref{prop:log-path} gives the inclusion
\[
 \{m_2(U)<x^2\}\subseteq\{d_A([I],[U])<xD\}.
\]
Together with the upper bound, this proves \eqref{eq:two-sided-speed}, after
increasing $C_x$ if necessary so that $C_x\ge c_x$.

Now let $x>\pi/\sqrt3$.  If $x>\pi$, the complement of the ball is empty by
\eqref{eq:log-path-upper}.  At $x=\pi$ it has Haar measure zero, since
$m_2(U)<\pi^2$ almost surely.  For $\pi/\sqrt3<x<\pi$, the same strict
inclusion yields
\[
 \{d_A([I],[U])\ge xD\}\subseteq\{m_2(U)\ge x^2\}.
\]
Set $\Delta_x=x^2-\pi^2/3>0$.  Applying
Proposition~\ref{prop:CUE-m2} at $\min\{\Delta_x,\Delta_{\rm C}\}$ gives
\eqref{eq:upper-threshold}.  The limit in \eqref{eq:zero-one-threshold}
follows from the two bounds.
\end{proof}

\subsection{Rate bounds near the threshold}

\begin{proposition}
\label{prop:vanishing-regularization}
Fix $0<x<\pi/\sqrt3$ and put
\[
 \Delta_x=\frac{\pi^2}{3}-x^2.
\]
For
\begin{equation}\label{eq:vanishing-regularizer}
 h_D=M^{1/2}D^{-1/4},
 \qquad r_D=e^{-h_D},
 \qquad \rho_D=\frac{1-r_D}{L_\star},
\end{equation}
one has
\begin{equation}\label{eq:vanishing-regularization-conclusion}
 \liminf_{\substack{D\to\infty\\D=2^n}}
 -\frac1{D^2}\log\mu_D\!\left(B_A([I],xD)\right)
 \ge\frac{\Delta_x^2}{16\zeta(3)}.
\end{equation}
\end{proposition}

\begin{proof}
The choice \eqref{eq:vanishing-regularizer} satisfies
$h_D\to0$, $\rho_D\asymp h_D$, and $\log(1/h_D)=o(D)$.  The comparison terms
in the proof of Proposition~\ref{prop:uniform-estimate} obey
\[
 \frac{\rho_D^{-1}MD^{3/2}}{D^2}
 =O(M^{1/2}D^{-1/4})=o(1),
 \qquad
 \frac{\rho_D^{-1}D\sqrt M}{D^2}=O(D^{-3/4})=o(1),
\]
while
\[
 \frac{MD^{3/2}}{D^2}=o(1),
 \qquad
 \frac{M\sqrt D}{D^2}=o(1),
 \qquad
 \frac{D\log(1/\rho_D)}{D^2}=o(1).
\]
The error from comparing the spherical averages satisfies
$\rho_D^{-1}\delta_D=o(1)$.  Moreover,
\[
 \frac1{D^2}\log\mathfrak C_D\to0,
 \qquad
 \frac1{D^2}\log\mathfrak B_D\to0,
 \qquad
 \frac1{D^2}\log\frac{x^{D-1}}{D-1}\to0.
\]
The full bad-event contribution remains $e^{-\Omega(D^{5/2})}$, since all
positive terms preceding its Gaussian tail are $o(D^{5/2})$ for the present
choice of $\rho_D$.

For every configuration retained in the full-sphere integral, the
principal-angle second moment is at most
$x^2$, so its deficit from $\pi^2/3$ is at least $\Delta_x$.  Applying
Lemma~\ref{lem:abel-sharp-coefficient} with \eqref{eq:vanishing-regularizer},
then repeating the Weyl cancellation and scalar integrations of
Proposition~\ref{prop:uniform-estimate}, gives
\[
 \limsup_{\substack{D\to\infty\\D=2^n}}
 \frac1{D^2}\log\mu_D\!\left(B_A([I],xD)\right)
 \le-\frac{\Delta_x^2}{16\zeta(3)}.
\]
This is \eqref{eq:vanishing-regularization-conclusion}.
\end{proof}

\begin{proof}[Proof of Corollary~\ref{cor:local-rate}]
Recall $x_\star=\pi/\sqrt3$ and write
\[
 \Delta_x=x_\star^2-x^2.
\]
Proposition~\ref{prop:vanishing-regularization} gives
\begin{equation}\label{eq:local-rate-geometric-lower}
 \underline I(x)\ge\frac{\Delta_x^2}{16\zeta(3)}.
\end{equation}

For the opposite inequality, Lemma~\ref{lem:CUE-energy-Fourier} and the Fourier
series of $\theta^2$ imply that every probability measure satisfying
\[
 \int_{\mathbb T}\vartheta\,d\nu
 \le\frac{\pi^2}{3}-\delta
\]
obeys
\begin{equation}\label{eq:CUE-rate-lower-quadratic}
 \mathcal I(\nu)\ge\frac{\delta^2}{16\zeta(3)}.
\end{equation}
Indeed,
\[
 \int_{\mathbb T}\vartheta\,d\nu-\frac{\pi^2}{3}
 =4\sum_{k\ge1}\frac{(-1)^k}{k^2}\Re\widehat\nu_k,
\]
and Cauchy--Schwarz in the norm of
\eqref{eq:CUE-rate-Fourier} gives \eqref{eq:CUE-rate-lower-quadratic}.

The coefficient is asymptotically attainable.  Let
$K_\delta=\lfloor e^{1/\delta}\rfloor$ and
$Z_K=\sum_{k=1}^Kk^{-3}$.  The density
\begin{equation}\label{eq:CUE-trial-density}
 \frac{d\nu_{\delta,K}}{d\theta}
 =\frac1{2\pi}\left[
 1+2\sum_{k=1}^{K}
 \frac{\delta(-1)^{k+1}}{4Z_Kk}\cos(k\theta)
 \right]
\end{equation}
is nonnegative for all sufficiently small $\delta$.  Indeed, its bracket is at
least $1-\delta H_K/(2Z_K)$, where
$H_K=\sum_{k\le K}k^{-1}$; for $K=K_\delta$ one has
$\delta H_K\le1+o(1)$ and $Z_K\to\zeta(3)$.  Direct calculation gives
\[
 \int_{\mathbb T}\vartheta\,d\nu_{\delta,K}
 =\frac{\pi^2}{3}-\delta
\]
and
\begin{equation}\label{eq:CUE-trial-energy}
 \mathcal I(\nu_{\delta,K})
 =\frac{\delta^2}{16Z_K}
 =\frac{\delta^2}{16\zeta(3)}(1+o(1))
 \qquad(\delta\downarrow0).
\end{equation}

Fix $\varepsilon>0$ and choose
$\delta=(1+\varepsilon)\Delta_x$.  The resulting trial measure has second
moment strictly below $x^2$.  Applying the open-set lower bound in
Theorem~\ref{thm:CUE-LDP} to a sufficiently small weak neighborhood of this
measure and using Proposition~\ref{prop:log-path} yields
\[
 \overline I(x)
 \le\frac{(1+\varepsilon)^2\Delta_x^2}{16\zeta(3)}(1+o(1))
 \qquad(x\uparrow x_\star).
\]
First take $D\to\infty$ in the definitions of the lower and upper rates.  Then
let $x\uparrow x_\star$, and finally let $\varepsilon\downarrow0$.  Combining the
resulting upper bound with \eqref{eq:local-rate-geometric-lower} proves
\eqref{eq:local-rate-coefficient}.
\end{proof}

\begin{corollary}
\label{cor:moments-logpath}
For every fixed $1\le p<\infty$ there is $C_p<\infty$ such that
\begin{equation}\label{eq:Lp-rate}
 \E\left|
 \frac{d_A([I],[U])}{D}-\frac{\pi}{\sqrt3}
 \right|^p
 \le C_p\eta_D^p
\end{equation}
for all sufficiently large $D$.  In particular,
\begin{equation}\label{eq:mean-distance}
 \E d_A([I],[U])
 =\frac{\pi}{\sqrt3}D+O(\eta_DD).
\end{equation}
For the path $\gamma_U$ of Proposition~\ref{prop:log-path},
\begin{equation}\label{eq:logpath-gap}
 \Prob\!\left\{
 0\le\ell_A(\gamma_U)-d_A([I],[U])
 >C\eta_DD
 \right\}
 \le2e^{-cM^{1/2}D^{7/4}}
\end{equation}
for suitable absolute $c,C>0$.  Consequently,
$\ell_A(\gamma_U)/d_A([I],[U])\to1$ in probability.
\end{corollary}

\begin{proof}
Let $\mathcal A_D$ be the good event
\[
 \mathcal A_D=\left\{
 \left|\frac{d_A([I],[U])}{D}-\frac{\pi}{\sqrt3}\right|
 \le C_{\rm sh}\eta_D\right\}.
\]
By \eqref{eq:sharp-concentration},
$\Prob(\mathcal A_D^c)\le2e^{-c_{\rm sh}M^{1/2}D^{7/4}}$.
On $\mathcal A_D$ the deviation is at most $C_{\rm sh}\eta_D$, and everywhere
it is bounded by $2\pi$ by \eqref{eq:log-path-upper}.  Thus
\[
 \E\left|\frac{d_A([I],[U])}{D}-\frac{\pi}{\sqrt3}\right|^p
 \le (C_{\rm sh}\eta_D)^p
      +2(2\pi)^p e^{-c_{\rm sh}M^{1/2}D^{7/4}},
\]
which proves \eqref{eq:Lp-rate}; its case $p=1$ gives
\eqref{eq:mean-distance}.
On the intersection of the lower-tail good event
\eqref{eq:sharp-lower-tail} and the CUE event
$m_2(U)\le\pi^2/3+\eta_D$, one has
\[
 \left(\frac{\pi}{\sqrt3}-C\eta_D\right)D
 \le d_A([I],[U])
 \le\ell_A(\gamma_U)
 \le\left(\frac{\pi}{\sqrt3}+C\eta_D\right)D.
\]
The two exceptional probabilities are bounded by the right-hand side of
\eqref{eq:logpath-gap}.  Since the common leading constant is positive, the
ratio conclusion follows.
\end{proof}

\section{Circuit bounds and further questions}\label{sec:circuit-outlook}

We compare $d_A$ with the projective operator-norm distance
$\delta_{\op}$ defined in Section~\ref{sec:model-results}.

\begin{proposition}\label{prop:projective-bridge}
For all $[U],[V]\in\PU(D)$,
\begin{equation}\label{eq:proj-bridge}
 d_A([U],[V])
 \le2D\arcsin\frac{\delta_{\op}([U],[V])}{2}.
\end{equation}
\end{proposition}

\begin{proof}
Set $\delta=\delta_{\op}([U],[V])$.  Choose $\phi\in\mathbb R$ so that
$\|I-e^{-i\phi}U^*V\|_{\op}=\delta$, and write
\[
 e^{-i\phi}U^*V=e^{i\Theta},
\]
where $\Theta$ is the principal Hermitian logarithm.  Let
$\theta_1,\ldots,\theta_D\in[-\pi,\pi]$ be its eigenvalues.  The
operator-norm bound implies
\[
 |\theta_j|\le \alpha:=2\arcsin(\delta/2),\qquad 1\le j\le D.
\]
Let $K=\Theta-D^{-1}\Tr(\Theta)I$.  Then $[V]=[Ue^{iK}]$, and the eigenvalues
of $K$ lie in an interval of length at most $2\alpha$.  The variance bound for
numbers in an interval gives
\begin{equation}\label{eq:centered-log-bound}
 \|K\|_0\le\alpha.
\end{equation}
Define $\gamma:[0,1]\to\PU(D)$ by $\gamma(t)=[Ue^{itK}]$.  Its
right-trivialized velocity is
$iUKU^*$, so
\[
 \|\dot\gamma(t)\|_A
 \le\sqrt{\lambda_{\max}(A)}\,\|UKU^*\|_0
 =D\|K\|_0.
\]
Integrating and using \eqref{eq:centered-log-bound} proves
\eqref{eq:proj-bridge}.
\end{proof}

\begin{lemma}\label{lem:gate-lengths}
An arbitrary projective two-qubit gate has Nielsen distance at most $\pi$ from
the identity.  A tensor product of $n$ one-qubit gates has distance at most
$\pi\sqrt n$.  Moreover, every no-ancilla circuit with $k$ arbitrary
two-qubit gates and arbitrary one-qubit gates can be written
\begin{equation}\label{eq:absorbed-factorization}
 V=L_{\rm out}G_k\cdots G_1,
\end{equation}
where every $G_j$ is an arbitrary two-qubit gate and $L_{\rm out}$ is one
local layer.  Consequently,
\begin{equation}\label{eq:circuit-upper}
 d_A([I],[V])\le\pi k+\pi\sqrt n.
\end{equation}
\end{lemma}

\begin{proof}
For a two-qubit gate acting on a fixed pair, choose a Hermitian logarithm
$K^{(2)}$ of a representative on $\mathbb C^4$ whose spectrum lies in an
interval of length at most $2\pi$, replace it by
$K^{(2)}-\tfrac14\Tr(K^{(2)})I_4$, and let $K$ denote its extension to the
$n$-qubit space.  Then $\|K\|_0\le\pi$.  Since $K$ has Pauli weight at most
two, the path $[e^{itK}]$ has length at most $\pi$.

For a local layer, center each one-qubit logarithm analogously by subtracting
half its trace and denote the resulting extensions by $K_1,\ldots,K_n$.
They commute, are Hilbert--Schmidt orthogonal, and satisfy
$\|K_j\|_0\le\pi$.  Therefore
\[
 \left\|\sum_{j=1}^nK_j\right\|_A^2
 =\sum_{j=1}^n\|K_j\|_0^2\le n\pi^2.
\]

To obtain \eqref{eq:absorbed-factorization}, process the circuit in temporal
order while retaining one pending local gate on each wire.  Immediately before
a two-qubit gate on wires $i,j$, multiply the two pending gates on those wires
into that arbitrary two-qubit gate.  Pending gates on all other wires commute
past it.  Iteration leaves the same number of two-qubit gates and one terminal
local layer.

Finally, let $F_1,\ldots,F_{k+1}$ be the factors in temporal order, with the
local layer last, set $V_0=I$, and define $V_j=F_jV_{j-1}$.  Right invariance gives
\[
 d_A([V_{j-1}],[V_j])=d_A([I],[F_j]).
\]
The triangle inequality and the two preceding length bounds prove
\eqref{eq:circuit-upper}.
\end{proof}

\begin{proof}[Proof of Corollary~\ref{cor:circuit}]
On the lower-tail event supplied by Theorem~\ref{thm:sharp}, let $V$ be a
circuit endpoint satisfying $\delta_{\op}([U],[V])\le\epsilon$.
Proposition~\ref{prop:projective-bridge}, Lemma~\ref{lem:gate-lengths}, and the
triangle inequality give
\[
 d_A([I],[U])
 \le\pi k+\pi\sqrt n+2D\arcsin\frac\epsilon2.
\]
Combining this with the lower side of \eqref{eq:sharp-concentration} and
dividing by $\pi$ gives \eqref{eq:circuit-main}.  For each fixed $\epsilon<\epsilon_*$, the coefficient of $D$ is
positive for all sufficiently large $D$.  The exceptional probability is
bounded as in \eqref{eq:circuit-probability}.
\end{proof}

\subsection*{Further questions}\label{sec:discussion}

The combination of the Jacobi lower bound and the CUE logarithmic upper bound
identifies the exact first-order Haar-typical distance:
\[
 d_A([I],[U])=\left(\frac{\pi}{\sqrt3}+o_{\Prob}(1)\right)D.
\]
The second-moment constant $\pi/\sqrt3$ is therefore the sharp
linear-radius threshold.

The rescaling in \eqref{eq:quarter-vars} balances the error from the moving
low-weight subspace against that from the off-diagonal commutator block.
More generally, rescaling by $T=G^\alpha$ gives leading pointwise error
terms of orders
\[
 q^\alpha MD
 \qquad\text{and}\qquad
 q^{-\alpha}M.
\]
For $q=D^{-2}$ their powers of $D$ are $1-2\alpha$ and $2\alpha$.  The maximum
of these exponents is minimized uniquely when
$1-2\alpha=2\alpha$, namely at $\alpha=1/4$.

Theorem~\ref{thm:output} applies to square blocks of time-dependent linear
systems.  It allows the comparison block to be singular, as occurs at
conjugate points.

Weyl's formula converts the regularized sinc product into a logarithmic-energy
problem on the circle.  The Abel estimate applies to either sign of the
second-moment deviation and also controls the CUE upper tail, giving the
same constant in the distance lower and upper bounds.

For Haar $U\in U(D)$, the centered principal-logarithm path
is asymptotically length minimizing to first order, as stated in
Corollary~\ref{cor:moments-logpath}.  Appendix~\ref{app:logpath-fluctuations}
records its order-one Gaussian fluctuations.
Determining the fluctuation scale of the true distance inside the present
$D^{7/8}(\log D)^{1/2}$ window remains open.

Theorem~\ref{thm:ball-probabilities} shows that $D^2$ is the correct
large-deviation speed below the threshold, and
Corollary~\ref{cor:local-rate} determines the near-threshold coefficient
$1/(16\zeta(3))$ as the threshold is approached.  A natural next problem is
the exact fixed-radius rate function
\[
 I(x)=\lim_{D\to\infty}-\frac1{D^2}
 \log\mu_D\!\left(B_A([I],xD)\right),
 \qquad 0<x<\frac{\pi}{\sqrt3}.
\]
After Weyl cancellation, the problem becomes a constrained
logarithmic-energy problem on the circle.  Identifying the geometric rate
function would additionally require a two-sided comparison between the
Jacobi determinant
and that of the constant-coefficient model.

Extending the argument to a variable cliff height $Q=Q_D$ would require
bounds for the errors from rescaling, transporting the low-weight subspace,
and omitting its momentum component.  The radial scaling and the spectral
bound along minimizing geodesics also depend on $Q_D$.  The negative term
in the Abel estimate would have to dominate all these comparison errors.

\appendix
\section{Fluctuations of the logarithmic path}
\label{app:logpath-fluctuations}

Throughout this appendix, $U$ is Haar distributed on $U(D)$ and
$\gamma_U$ is the centered principal-logarithm path from
Proposition~\ref{prop:log-path}.  The path depends on the
representative $U$, whereas the distance depends only on its projective
class in $\PU(D)$.  We write
$\mathcal N(0,\sigma^2)$ for the centered normal distribution with variance
$\sigma^2$.

\begin{lemma}\label{lem:principal-angle-statistics}
Let $\theta_1,\ldots,\theta_D\in(-\pi,\pi]$ be the principal eigenangles of
$U$, and set
\[
 S_2=\sum_{j=1}^D\theta_j^2,
 \qquad
 S_1=\sum_{j=1}^D\theta_j.
\]
Then
\begin{equation}\label{eq:CUE-m2-CLT}
 S_2-D\frac{\pi^2}{3}
 \ \Longrightarrow\ \mathcal N(0,8\zeta(3)),
\end{equation}
and
\begin{equation}\label{eq:S1-variance}
 \E S_1^2
 =2\sum_{k\ge1}\frac{\min\{k,D\}}{k^2}
 =O(\log D).
\end{equation}
\end{lemma}

\begin{proof}
We use the Fourier convention
\[
 \widehat f_k=\frac1{2\pi}\int_{-\pi}^{\pi}
 f(\theta)e^{-ik\theta}\,d\theta.
\]
For $f_2(\theta)=\theta^2$ on $(-\pi,\pi]$, the Fourier coefficients are
\[
 \widehat f_{2,0}=\frac{\pi^2}{3},
 \qquad
 \widehat f_{2,k}=\frac{2(-1)^k}{k^2},\quad k\ne0.
\]
Hence
\[
 \sum_{k\ne0}|k|\,|\widehat f_{2,k}|^2=8\zeta(3)<\infty.
\]
The CUE linear-statistic theorem of Diaconis and Evans
\cite[Theorem~5.1]{DiaconisEvans} therefore gives
\eqref{eq:CUE-m2-CLT} with variance $8\zeta(3)$.  The same limit also follows
from the circular $\beta$-ensemble ($C\beta E$) central limit theorem of
Feng, Tian, and Wei \cite[Theorem~1.2]{FengTianWei}, applied at $\beta=2$ to
the centered periodic function $f_2-\pi^2/3\in W^{1,p}(S^1)$,
$1<p<\infty$.

For $f_1(\theta)=\theta$, interpreted as its $2\pi$-periodic $L^2$ function,
$|\widehat f_{1,k}|=1/|k|$ for $k\ne0$.  Applying the exact CUE covariance
identity \cite[Theorem~2.1(b)]{DiaconisEvans}
\[
 \E\!\left[\Tr(U^k)\overline{\Tr(U^\ell)}\right]
 =\delta_{k\ell}\min\{k,D\},\qquad k,\ell\ge1
\]
to the truncated Fourier series of $f_1$ and then passing to the $L^2$ limit
gives \eqref{eq:S1-variance}.  In particular, $S_1^2/D\to0$ in probability.
\end{proof}

\begin{lemma}\label{lem:adjoint-isotropy}
Let $H$ be a fixed traceless Hermitian matrix and let $V$ be Haar distributed
on $U(D)$.  Then
\begin{equation}\label{eq:adjoint-isotropy}
 \E_V\|P(iVHV^*)\|_0^2
 =\frac{M}{D^2-1}\|H\|_0^2.
\end{equation}
The same identity holds conditionally on the eigenangles of a Haar unitary for
$H=\diag(\theta_1-\bar\theta,\ldots,\theta_D-\bar\theta)$, where
$\bar\theta=D^{-1}\sum_j\theta_j$.
\end{lemma}

\begin{proof}
With $dV$ denoting normalized Haar measure, the averaged operator
\[
 \mathcal A=\int_{U(D)}\Ad_{V}^{-1}P\Ad_V\,dV
\]
commutes with the adjoint action on the real Hilbert space of traceless
skew-Hermitian matrices.  This representation is irreducible, so Schur's
lemma gives $\mathcal A=cI$.  Taking traces yields
$c=\rank(P)/(D^2-1)=M/(D^2-1)$, which proves
\eqref{eq:adjoint-isotropy}.  Conditional on a simple spectrum, the eigenvector
matrix is Haar on the corresponding conjugacy orbit; the repeated-spectrum
set has Haar measure zero.
\end{proof}

\begin{proposition}
\label{prop:logpath-CLT}
For Haar $U\in U(D)$, the path $\gamma_U$ satisfies
\begin{equation}\label{eq:logpath-CLT}
 \ell_A(\gamma_U)-\frac{\pi}{\sqrt3}D
 \ \Longrightarrow\
 \mathcal N\!\left(0,\frac{6\zeta(3)}{\pi^2}\right).
\end{equation}
\end{proposition}

\begin{proof}
With
\[
 v_D=\|K_U\|_0^2
 =\frac{S_2}{D}-\left(\frac{S_1}{D}\right)^2,
\]
Lemma~\ref{lem:principal-angle-statistics} gives
\[
 D\left(v_D-\frac{\pi^2}{3}\right)
 \ \Longrightarrow\ \mathcal N(0,8\zeta(3)).
\]
The delta method therefore yields
\begin{equation}\label{eq:HS-logpath-CLT}
 D\sqrt{v_D}-\frac{\pi}{\sqrt3}D
 \ \Longrightarrow\
 \mathcal N\!\left(0,\frac{6\zeta(3)}{\pi^2}\right).
\end{equation}

It remains to compare the Hilbert--Schmidt upper bound with the path
length.  Since the path velocity is $iK_U$,
\[
 \ell_A(\gamma_U)^2
 =D^2\|K_U\|_0^2-(D^2-1)\|P(iK_U)\|_0^2.
\]
Conditional on the eigenangles, Lemma~\ref{lem:adjoint-isotropy} gives
\[
 \E\!\left(\|P(iK_U)\|_0^2\mid\theta_1,\ldots,\theta_D\right)
 =\frac{M}{D^2-1}\|K_U\|_0^2.
\]
If $K_U\ne0$, the inequality $1-\sqrt{1-t}\le t$ gives
\[
 0\le D\|K_U\|_0-\ell_A(\gamma_U)
 \le D\|K_U\|_0
 \frac{D^2-1}{D^2}
 \frac{\|P(iK_U)\|_0^2}{\|K_U\|_0^2};
\]
the case $K_U=0$ is immediate.  Since $\|K_U\|_0\le\pi$, conditional
expectation yields
\[
 \E\bigl|D\|K_U\|_0-\ell_A(\gamma_U)\bigr|
 \le\frac{\pi M}{D}=o(1).
\]
Combining this estimate with \eqref{eq:HS-logpath-CLT} proves
\eqref{eq:logpath-CLT}.
\end{proof}


\begin{thebibliography}{99}

\bibitem{NielsenQIC}
M.~A. Nielsen,
\emph{A geometric approach to quantum circuit lower bounds},
Quantum Inf. Comput. \textbf{6} (2006), 213--262.

\bibitem{NielsenScience}
M.~A. Nielsen, M.~R. Dowling, M.~Gu, and A.~C. Doherty,
\emph{Quantum computation as geometry},
Science \textbf{311} (2006), 1133--1135.

\bibitem{NielsenPRA}
M.~A. Nielsen, M.~R. Dowling, M.~Gu, and A.~C. Doherty,
\emph{Optimal control, geometry, and quantum computing},
Phys. Rev. A \textbf{73} (2006), 062323.

\bibitem{DowlingNielsen}
M.~R. Dowling and M.~A. Nielsen,
\emph{The geometry of quantum computation},
Quantum Inf. Comput. \textbf{8} (2008), 861--899.

\bibitem{BrownBG}
A.~R. Brown,
\emph{A quantum complexity lower bound from differential geometry},
Nature Phys. \textbf{19} (2023), 401--406;
arXiv:2112.05724.

\bibitem{BrownPoly}
A.~R. Brown,
\emph{Polynomial equivalence of complexity geometries},
Quantum \textbf{8} (2024), 1391.

\bibitem{YuanCircuit}
P.~Yuan, J.~Allcock, and S.~Zhang,
\emph{Does qubit connectivity impact quantum circuit complexity?},
IEEE Trans. Comput.-Aided Des. Integr. Circuits Syst. \textbf{43} (2024),
520--533.

\bibitem{DiaconisEvans}
P.~Diaconis and S.~N. Evans,
\emph{Linear functionals of eigenvalues of random matrices},
Trans. Amer. Math. Soc. \textbf{353} (2001), 2615--2633.

\bibitem{HiaiPetz}
F.~Hiai and D.~Petz,
\emph{A large deviation theorem for the empirical eigenvalue distribution of
random unitary matrices},
Ann. Inst. H. Poincar\'e Probab. Statist. \textbf{36} (2000), no.~1,
71--85.

\bibitem{RibeiroTrancanelli}
M.~Rios Ribeiro and D.~Trancanelli,
\emph{Nielsen complexity with multiple cost factors},
arXiv:2606.02817 (2026).

\bibitem{LBLP}
A.~Le Brigant, L.~Lichtenfelz, and S.~C. Preston,
\emph{Conjugate points on Lie groups with left-invariant metrics},
arXiv:2408.03854v3 (2025).

\bibitem{Kato}
T.~Kato,
\emph{Perturbation Theory for Linear Operators},
Springer, Berlin, 1995.

\bibitem{Bhatia}
R.~Bhatia,
\emph{Matrix Analysis},
Springer, New York, 1997.

\bibitem{Hall}
B.~C. Hall,
\emph{Lie Groups, Lie Algebras, and Representations: An Elementary
Introduction}, 2nd ed., Springer, Cham, 2015.

\bibitem{Macdonald}
I.~G. Macdonald,
\emph{The volume of a compact Lie group},
Invent. Math. \textbf{56} (1980), 93--95.

\bibitem{Mehta}
M.~L. Mehta,
\emph{Random Matrices}, 3rd ed., Elsevier, Amsterdam, 2004.

\bibitem{Chavel}
I.~Chavel,
\emph{Riemannian Geometry: A Modern Introduction}, 2nd ed.,
Cambridge University Press, Cambridge, 2006.

\bibitem{FengTianWei}
R.~Feng, G.~Tian, and D.~Wei,
\emph{The Berry--Esseen theorem for circular $\beta$-ensemble},
Ann. Appl. Probab. \textbf{33} (2023), no.~6B, 5050--5070.

\end{thebibliography}
\end{document}